\documentclass [12pt]{amsart}
\usepackage{amssymb,latexsym,amsfonts, amsthm, amsmath, amsrefs, mathtools,color}
\usepackage{array}
\usepackage{enumerate}
\usepackage{booktabs}
\usepackage{fancyhdr}
\usepackage{verbatim}
\usepackage[margin=1in,twoside=false]{geometry}
\usepackage{multicol}
\usepackage{longtable}
\usepackage{hyperref}
\usepackage[capitalize]{cleveref}
\usepackage{mathrsfs}
\usepackage{tikz}
\usepackage{accents}
\usetikzlibrary{matrix,arrows}
\usepackage{algpseudocode}
\usepackage[section]{algorithm}

\usepackage{breqn}

\theoremstyle{definition}
\newtheorem{thm}[algorithm]{Theorem}
\newtheorem*{thm*}{Theorem}
\newtheorem{ithm}{Theorem}

\newtheorem{problem}[algorithm]{Problem}

\newtheorem{lem}[algorithm]{Lemma}
\newtheorem{prop}[algorithm]{Proposition}
\newtheorem{defn}[algorithm]{Definition}
\newtheorem{ex}[algorithm]{Example}
\newtheorem{obs}[algorithm]{Observation}
\newtheorem{question}[algorithm]{Question}
\newtheorem*{question*}{Question}

\newcommand{\cC}{\mathcal C}

\newcommand{\cT}{\mathcal T}

\newcommand{\cF}{\mathcal F}
\newcommand{\cV}{\mathcal V}
\newcommand{\cH}{\mathcal H}

\newcommand{\cP}{\mathcal P}

\newcommand{\cM}{\mathcal M}

\newcommand{\VC}{\mathrm{VC}}

\newcommand{\R}{\mathbb R}

\DeclareMathOperator{\conv}{conv}

\DeclareMathOperator{\sign}{sign}

\DeclareMathOperator{\diff}{diff}
\DeclareMathOperator{\thresh}{thresh}

\DeclareMathOperator{\monr}{mRank}
\DeclareMathOperator{\mrank}{mRank}
\DeclareMathOperator{\vcrank}{vcRank}
\DeclareMathOperator{\orank}{oRank}
\DeclareMathOperator{\omrank}{oRank}

\DeclareMathOperator{\rank}{rank}
\DeclareMathOperator{\radr}{radRank}

\DeclareMathOperator{\sep}{sep}

\newcommand{\supp}[1]{\underline{#1}}

\newcommand{\inv}{^{-1}}

\renewcommand{\emptyset}{\varnothing}

\title{Using oriented matroids to find low rank structure in presence of nonlinearity}
\author{Caitlin Lienkaemper}

\begin{document}

\begin{abstract}
Estimating the linear dimensionality of a data set in the presence of noise is a common problem.
However, data may also be corrupted by monotone nonlinear distortion that preserves the ordering of matrix entries but causes linear methods for estimating rank to fail. 
In light of this, we consider the problem of computing  \emph{underlying rank}, which is the lowest rank consistent with the ordering of matrix entries, and \emph{monotone rank}, which is  the lowest rank consistent with the ordering within columns. 
We show that each matrix of monotone rank $d$ corresponds to a point arrangement and a hyperplane arrangement in $\R^{d}$, and that the ordering within columns of the matrix can be used to recover information about these arrangements. 
Using Radon's theorem and the related concept of the VC dimension, we can obtain lower bounds on the monotone rank of a matrix. 
However, we also show that the monotone rank of a matrix can exceed these bounds.
In order to obtain better bounds on monotone rank, we develop the connection between monotone rank estimation and oriented matroid theory. 
Using this connection, we show that monotone rank is difficult to compute: the problem of deciding whether a matrix has monotone rank two is already NP-hard. 
However, we introduce an ``oriented matroid completion" problem as a combinatorial relaxation of the monotone rank problem and show that checking whether a set of sign vectors has matroid completion rank two is easy. 
 \end{abstract}

\maketitle

\section{Introduction}	

Estimating the dimensionality of a data set is a common problem, particularly in computational neuroscience and mathematical biology more broadly \cites{altan2021estimating, cunningham2014dimensionality, stringer2019high, fusi2016neurons}.
This is typically done using  singular value decomposition, which can estimate the rank of a matrix even in the presence of noise. 
However, measurements may be corrupted by a monotone nonlinear distortion, rather than additive noise \cites{curto22novel, giusti2015clique,   weinreich2006darwinian, husain2020physical, otwinowski2018biophysical}. 
For instance, this can occur when the measured variable is a proxy for some underlying value, such as calcium fluorescence for neural activity \cite{ ak2012opt,siegle2021reconciling,huang2021relationship}.
This can distort the spectrum of a matrix, making estimates of rank based on singular values unreliable. 

\begin{ex}\label{ex:distortion}
The matrix 
\begin{align*}
B  = \begin{pmatrix}13.01 & 12.4& 0.08 \\ 1.6 & 8.52 & -5.56 \\ 2.06 & -3.23 & 4.14 \\ 17.48 & 25.26 & -6.74\end{pmatrix}
\end{align*}
has  singular values $( 37.01, 8.94,  1.07\times 10^{-15})$. Since only two of these are non-negligible, we estimate that $B$ is rank two. 
However, the matrix
\begin{align*}
A = \begin{pmatrix}3.67 & 3.46 & 1.01 \\ 1.17 & 2.34 & 0.57 \\ 1.23 & 0.72 & 1.51 \\ 5.74 & 12.5 & 0.51\end{pmatrix}
\end{align*}
with $A_{ij} = f(B_{ij})$, where here $f(x) = e^{x/10}$,
has three non-negligible singular values $(14.86,\allowbreak 2.42,\allowbreak 0.88)$, and therefore we estimate its rank as three. 
\end{ex}

Thus, singular values are not able to detect the underlying rank of a matrix in the presence of a monotone, nonlinear distortion. We would like a method that would return an estimate of rank two for the matrix A, matching the rank of the underlying matrix B.

Monotone nonlinear functions preserve some information about the matrix, namely the order of  the entries. 
In this paper, we consider how to use this combinatorial information to estimate the rank of a matrix that may be distorted by a monotone nonlinearity. Specifically, we consider the {\it underlying rank} of a matrix, defined in [10], which is the smallest rank consistent with the ordering of matrix entries. The underlying rank serves as a proxy for the true rank of the matrix before the transformation.

\begin{defn}
The \emph{underlying rank} of an $m\times n$ matrix $A$  is the smallest rank $d$ such that there exists a rank $d$ matrix $B$ and a monotone function $f$ such that $A_{ij} = f(B_{ij})$ for all $i = 1, \ldots, m$, $j = 1, \ldots, n$. 
\end{defn}

For instance, the matrix $A$ in Example \ref{ex:distortion} has rank three, but has underlying rank at most two because it is produced from the matrix $B$ by a monotone function, and thus its entries are in the same order.

As it turns out, all of our techniques for estimating underlying rank depend only on either the ordering within columns or within rows of the matrix. Thus we can consider instead the notion of \emph{monotone rank}, introduced in \cite{egger20xxtopological}, which allows us to take a different monotone nonlinear function for each column (or row).

\begin{defn}The \emph{monotone rank} $\mrank(A)$ of an $m\times n$ matrix $A$ is the smallest rank $d$ such that there exists a rank $d$ matrix $B$ and monotone functions $f_1, \ldots, f_n$ such that $A_{ij} = f_j(B_{ij})$ for all $i = 1, \ldots, m$, $j = 1, \ldots, n$. 
\end{defn}

Note that the monotone rank is the lowest rank consistent with the ordering of entries within columns, while the underlying rank is the lowest rank consistent with overall ordering of entries.  
Therefore monotone rank is a lower bound for underlying rank; there exist examples where the underlying rank is greater \cite{curto22novel}.
Notice that $\mrank(A^T)$ is the lowest rank consistent with the ordering of entries within \emph{rows} of a matrix $A$. In contrast to the ordinary rank, $\mrank(A)$ and $\mrank(A^T)$ may be different. 

\textbf{
Given a matrix $A$, how can we find its monotone rank?}
Monotone rank estimation falls into a broader body of work seeking low-dimensional structure in combinatorial data, which we review here. Giusti et al. \cite{giusti2015clique} introduce the problem of seeking geometric structure in matrices which is invariant under monotone transformation, and approach this problem using topological data analysis. The monotone rank was first formally defined by  Egger et al. \cite{egger20xxtopological}, which uses topological methods to estimate underlying rank. Curto et al. introduce the underlying rank in \cite{curto22novel}, which presents methods for estimating the underlying rank with applications to neural data. In \cite{curto2021betti}, Curto et al. characterize the Betti curves of matrices of underlying rank one. 
Finally, in \cite{dunn2018signed}, Anderson and Dunn discuss a similar problem inspired by psychology: they use oriented matroid theory to determine whether the data corrupted by an unknown monotone transformation is consistent with a given linear model. 

Just as the underlying rank of a matrix gives the minimal rank consistent with the order of matrix entries, the \emph{sign rank} of a matrix is the minimal rank consistent with the sign pattern of a matrix. The sign rank has been broadly studied in theoretical computer science \cite{alon2014sign, paturi1986probabilistic, forster2002linear}. In \cite{basri2009visibility}, Basri et al. consider a problem from computer vision, testing whether a set of images is compatible with a single three dimensional object. They prove that this problem is equivalent to checking whether a matrix has sign rank three, and that checking whether a matrix has sign rank three is already NP-complete, a result proven independently in \cite{bhangale2015complexity}.

In this paper, we connect the monotone rank to point arrangements, hyperplane arrangements, and oriented matroids with the goal of computing lower bounds on monotone rank. 
Using this connection, we introduce the \emph{Radon rank} and the \emph{VC rank} of a matrix, and prove that these are both lower bounds for monotone rank. 

\begin{ithm}
\label{ithm:rad_vc}
For any matrix $A$, the Radon rank and the VC rank are both lower bounds for the monotone rank. 
\end{ithm}

We use a connection to the \emph{sign rank} to  show that the monotone rank of a matrix can grow like the square root of matrix size by translating an example from \cite{forster2002linear}. 
This is faster than either the Radon rank or the VC rank can grow. 
\begin{ithm}
\label{ithm:sqrt}
For each $N = 2^n$,  there exists a  $N\times N$ matrix $A_N$ with $$\monr(A_{N}) \geq  \sqrt{N}-1.$$
\end{ithm}

Next, we connect monotone rank to oriented matroid theory. 
We define two notions of the \emph{oriented matroid rank}, the \emph{threshold oriented matroid rank} which generalizes the Radon rank and the \emph{difference oriented matroid rank} which generalizes the VC rank. We define the oriented matroid rank to be the maximum of these two values. 
We show that they are both lower bounds for the monotone rank, and can exceed the bounds from Radon rank and VC rank.

\begin{ithm}
\label{ithm:omrank}
For any matrix $A$, the oriented matroid rank is a lower bound for the monotone rank. 
\end{ithm}

This connection to oriented matroid theory also allows us to prove that computing monotone rank is hard, even when restricting to rank two:

\begin{ithm}
\label{ithm:hard}
The problem of deciding whether a matrix $A$ has monotone rank two is complete for the existential theory of the reals, and therefore NP-hard. 
\end{ithm}

On the other hand, determining whether a matrix has oriented matroid rank two is easy: 

\begin{ithm}
\label{ithm:easy}
There is an $O(m^2n)$ algorithm to determine whether an $m\times n$ matrix has oriented matroid rank two. 
\end{ithm}

\setcounter{ithm}{0}
Along the way, we introduce the \emph{oriented matroid completion} problem as a common combinatorial relaxation of the monotone rank problem, the sign rank problem, and the underlying rank problem. 

\begin{problem}\label{prob:oriented_matroid_completion}
Given a set of sign vectors, what is the minimal rank of an oriented matroid $\cM$ which contains them among its topes?
\end{problem}
Progress on this oriented matroid completion problem translates to progress on the monontone rank problem: lower bounds on the oriented matroid completion rank of a set of sign vectors translate to lower bounds on monotone rank. 

The structure of this paper is as follows. 
In Section \ref{sec:geometry}, we introduce the point  and hyperplane arrangements associated to each matrix. 
In Section \ref{sec:sign_vectors}, we associate two sets of sign vectors to a matrix. We then can define the Radon rank and the VC rank and prove Theorem \ref{ithm:rad_vc}. 
We then discuss the connection to sign rank and prove Theorem \ref{ithm:sqrt}. 
In Section \ref{sec:matroids}, we introduce oriented matroid theory, define the oriented matroid rank, and prove  Theorem \ref{ithm:omrank}. We also discuss the oriented matroid completion problem. 
In Section \ref{sec:complexity}, we consider the computational complexity of monotone rank and  oriented matroid rank, focusing specifically on the rank two case. 
We prove Theorems \ref{ithm:hard} and \ref{ithm:easy}. 
Finally, we conclude with open questions in Section \ref{sec:conclusion}. 

\section{Column permutations and the geometry of monotone rank}\label{sec:geometry}
Our approach depends on connecting matrices of monotone rank $d$ to point arrangements and hyperplane arrangements in $\R^{d}$, 
and using the ordering within columns to recover combinatorial information about these arrangements. 
This results from a standard result in linear algebra, the existence of the \emph{rank factorization} \cite{piziak1999full}. 	
\begin{lem}[\cite{piziak1999full}]
If $B$ is a $m\times n$ matrix of rank $d$, then $B$ has a rank factorization $B = PH^T$, where $P$ is $m\times d$ and $H$ is $n\times d$. 
\end{lem}

By considering the rows of $P$ and $H$ as vectors in $\R^d$,
we apply this to matrices of \emph{monotone rank} $d$ as follows: 
\begin{obs}
If $A$ is an  $m\times n$ matrix of monotone rank $d$, there exist sets of vectors $\cP = \{p_1, \ldots, p_m\}, \allowbreak \cH = \{h_1, \ldots, h_n\} \subset \R^d$ and monotone functions $\cF = \{f_1, \ldots, f_n\}$  such that $A_{ij} = f_j(p_i \cdot h_j)$. 
\end{obs}

We say that $(\cP = \{p_1, \ldots, p_m\}, \cH = \{h_1, \ldots, h_m\}, \cF = \{f_1, \ldots, f_n\})$ is a \emph{rank $d$ representation} of $A$. 
Based on Observation \ref{obs:sweep_orders}, we will call  $\cP$ the set of  \emph{points} and think of it as representing a point arrangement, and  will call $\cH$ the set of \emph{normal vectors} and think of it as representing a hyperplane arrangement. The roles of $\cP, \cH,$ and $\cF$ are illustrated in Figure \ref{fig:matrix_and_point_arrs}. 

\begin{figure}[ht!]
\includegraphics[width = 6 in]{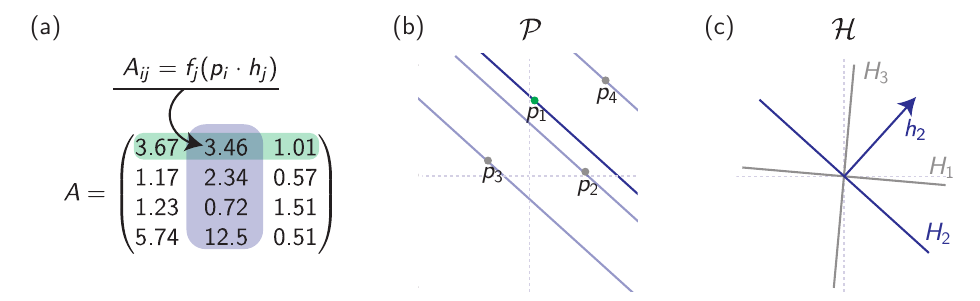}
\caption{\label{fig:matrix_and_point_arrs}
(a) The matrix $A$ from Example \ref{ex:distortion}. Since $A$ has monotone rank two, by Observation \ref{obs:sweep_orders}  there exist points $\cP = \{p_1, p_2, p_3, p_4 \}\subset \R^2$, hyperplane normals $\cH = \{h_1, h_2, h_3\} \subset \R^2$, and monotone functions $\cF = \{f_1, f_2, f_3\}$ such that  $A_{ij} = f_j(p_i \cdot h_j)$. 
(b) The point arrangement $\cP$. The order of entries in the second column of $B$ matches the order in which a hyperplane with normal vector $h_2$  sweeps past the points $p_1, p_2, p_3, p_4$. 
(c) The hyperplane arrangement $\cH$. The hyperplanes $H_1, H_2,$ and $H_3$ have normal vectors $h_1, h_2,$ and $h_3$. 
}

\end{figure}

%

We here review the basics of hyperplane arrangements. Recall that any vector $h\in \R^{d}$ defines a hyperplane $H = \{v \in \R^{d} \mid h \cdot v = 0\}$ and two half-spaces, $H^+  = \{v \in \R^{d} \mid h \cdot v > 0\}$ and  $H^-  = \{v \in \R^{d} \mid h \cdot v < 0\}$.  
Notice that $H$ must contain the origin. 
A (\emph{central) hyperplane arrangement}  is an arrangement of hyperplanes of this form, all of which must intersect at the origin.
An affine hyperplane is of the form $H = \{v\in \R^d\mid  v\cdot h = \theta\}$. Notice that each normal vector $h$ gives rise to a family of parallel affine hyperplanes. 
%

We can use the ordering of matrix entries to recover combinatorial information about the point arrangement $\cP$ and the hyperplane arrangement $\cH$, which we can use to estimate monotone  rank.
Following \cite{goodman1993allowable}, we define the \emph{sweep permutations} of a point arrangement, which we illustrate in Figure \ref{fig:sweep_perms}. 

\begin{defn}
Let $\cP  \subset \R^d$ be a set of points. Say $h \in \R^d$ is generic with respect to $\cP$ if  each hyperplane with normal vector $h$ contains at most one point of $\cP$. For $h$ generic with respect to $V$, 
define  $\pi_h$ to be the permutation which sends $i$ to the $i^{th}$ element encountered by a hyperplane swept in the direction of $h$. 
Define the set of \emph{sweep permutations} of $\cP$ to be the set \[\Pi(\cP) = \{\pi_h \mid h\in \R^d \mbox{ generic with respect to } \cP\}\]
\end{defn}

\begin{figure}[ht!]
\includegraphics[width = 6 in]{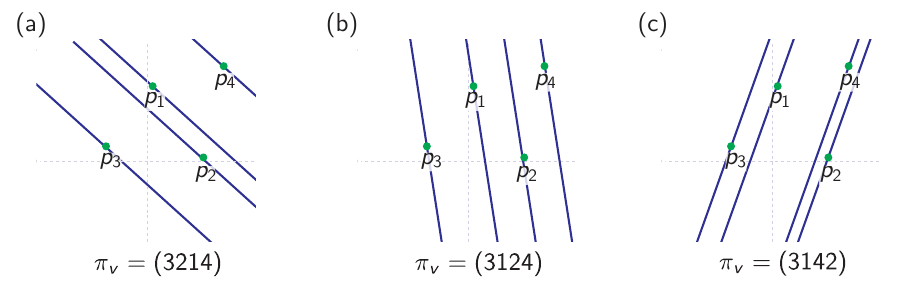}
\caption{\label{fig:sweep_perms}
Rotating a vector $v$ to obtain different sweep permutations of $\cP$. (a) The sweep permutation $\pi_v = (3214)$. (b) The sweep permutation $\pi_v = (3124)$. (c) The sweep permutation $\pi_v = (3124)$. }
\end{figure}

It is possible to recover a considerable amount of information about a point configuration $\cP$, including its dimension, from $\Pi(\cP)$ \cite{padrol2021sweeps}.  
Further, we can use the column orders of a matrix $A$ to recover a subset of $\Pi(\cP)$--our goal is to use this information to bound the rank of $A$.

\begin{obs}\label{obs:sweep_orders}
If $(\cP,\cH, \cF)$ is a rank $d$ representation of $A$, 
the order of entries within the $j^{th}$  column of $A$ is $\pi_{h_j}$, i.e. the order in which a hyperplane swept perpendicular to $h_j$ encounters the points $p_1, \ldots, p_n$. 
\end{obs}

Based on this observation, we define $\Pi(A)$ for a matrix: 

\begin{defn}
Let $A$ be a $m\times n$ matrix. For each column $j$ of the matrix, define a permutation $\pi_j$ which sends $i$ to the $i^{th}$ smallest entry of the $j^{th}$ column of $A$. Define $\Pi(A) = \{\pi_j \mid j\in [n]\}$. 
\end{defn}

Combining these definitions with our observation, we have $\Pi(A) \subseteq \Pi(\cP)$  whenever $\cP$ is the point arrangement of a rank $d$ representation of $A$. 
Thus  a matrix $A$  has monotone rank $d$ if and only if there exists a point arrangement $\cP$ in $\R^d$ such that $\Pi(A) \subseteq \Pi(\cP)$.
We can rephrase the monotone rank estimation problem as follows:
\begin{problem}
Given a set of permutations $\Pi(A)$ arising from a matrix, what is the minimal dimension of a point configuration $\cP$ such that $\Pi(A) \subseteq \Pi(\cP)$? 
\end{problem}

In order to make progress on this problem, in the next section we will encode the column permutations $\Pi(A) $ as sign vectors.

\section{Sign vectors and VC dimension}
\label{sec:sign_vectors}
Two simple bounds on monotone rank arise from Radon's theorem, and  are closely related to the VC dimension from statistical learning theory \cite{vapnik1971uniform}. 
In order to state this bound, we first give two ways to encode the order information of a matrix into a set of sign vectors. 
Here, a sign vector is an element of $\{+, 0, -\}^m$. 
A sign vector $X$ can also be written as a signed set $(X^+, X^-)$ where $X^+ = \{i \mid X_i = +\}$, $X^- = \{j \mid X_j = -\}$. 
Given either a point arrangement or a hyperplane arrangement, we can obtain a set of sign vectors known as the \emph{topes} of the arrangement. 
Our notation is inspired by oriented matroid theory, which we will review in Section \ref{sec:matroids}. 

\begin{defn}
The \emph{topes} of a point arrangement $\cP$,  written $\cT(\cP)$, are the sign vectors which arise as

\begin{align*}
(\sigma_H)_i = \begin{cases} 
+ & p_i \in H^+\\
- & p_i \in H^-
\end{cases}
\end{align*}
for a for a hyperplane $\cH$ containing no point of $\cP$.

The \emph{topes} of a central hyperplane arrangement $\cH$, written $\cT(\cH)$, are the sign vectors which arise as 
\begin{align*}
(\sigma_w)_i = \begin{cases} 
+ & w\in H_i^+\\
- &w \in H_i^-
\end{cases}
\end{align*}
for a point $w$ which does not lie on any hyperplane of $\cH$.
\end{defn}

The topes of an arrangement recover considerable geometric information, including the dimension, as we will show next. 
After that, we will show that we can recover subsets of $\cT(\cH)$ and $\cT(\cP)$ from $A$. 
We can then use the combinatorial properties of these subsets to bound the monotone rank of $A$. 

\subsection{Radon's theorem and VC dimension}

The classic version of Radon's theorem states that if $\cP  \subset \R^d$ is a set of $m$ points, $m \geq d + 2$, then $\cP$ has a  Radon partition: disjoint subsets whose convex hulls have a nonempty intersection \cite{radon1921mengen}. 
Here, we make use of two equivalent statements, and include a proof since our version diverges from the standard one. 

\begin{thm*}[Radon's theorem] 
\qquad
\begin{enumerate}[(a)]
\item 
A point arrangement $\cP$ is affinely independent if and only if for each disjoint $\cP_1, \cP_2 \subset \cP$, there is an affine hyperplane which separating the points in $\cP_1$ from the points in $\cP_2$. 

\item A hyperplane arrangement $\cH$ has linearly independent normal vectors if and only if for each disjoint $\cH_1, \cH_2 \subset \cH$, there is a point which is on the positive side of all hyperplanes of $\cH_1$ and on the negative side of  all hyperplanes of $\cH_2$. 
\end{enumerate}
\end{thm*} 
\begin{proof}
We first prove that a set of vectors  $\cV = \{v_1, \ldots, v_m\} \subseteq \R^{d}$ is linearly independent if and only if for all disjoint $\cV_1, \cV_2 \subseteq \cV$, there exists a vector $w\in \R^d$ such that $w \cdot v> 0$ for all $v \in \cV_1$, $w\cdot v < 0$ for all $v \in \cV_2$. 
We will then show this is equivalent to both (a) and (b). 

First, suppose $\cV$ is linearly independent. Let $V$ be the matrix with rows $v_1, \ldots, v_m$. 
Then the linear map $\R^d\to \R^m$ given by the matrix $V$ has image $\R^m$. 
Then for all $x\in \R^m$, there exists $w\in \R^d$ such that $V w = x$.  
Each coordinate of $x$ is equal to $v\cdot w$ for some $v\in \cV$.  
Thus for any disjoint $\cV_1, \cV_2 \subseteq \cV$, we can find a $w$ so that $w \cdot v > 0$ for $v \in \cV_1$ and $w\cdot v < 0 $ for all $v \in \cV_2$. 
Now, suppose that  $\cV$ is not  linearly independent.
Then there exist $\lambda_1, \ldots, \lambda_m$, not all zero, such that $\lambda_1 v_1 + \cdots+ \lambda_m  v_m = 0$. 
Let $\cV_1 = \{v \mid \lambda_i > 0\}$, $\cV_2 = \cV \setminus \cV_1$. 
Then $\sum_{v_i \in \cV_1} \lambda_i v_i = \sum_{v_j \in \cV_2} (-\lambda_j) v_j = p $. 
That is, there exists a point $p$ which we can express as a positive linear combination of both  $\cV_1$ and $\cV_2$. 
 Now, let $w \in \R^d$. If $w \cdot v> 0$ for all $v\in \cV_1$, then $w\cdot p > 0$. 
 Thus we cannot have $w \cdot v< 0$ for all $v\in \cV_2$, since this would imply $w\cdot p < 0$. 

Now, we use this fact to prove (a). 
Let  $\cP = \{p_1, \ldots, p_m\}\subseteq \R^d$. Let $\hat \cP = \{\hat p_1, \ldots, \hat p_m\} \subseteq \R^{d+1}$ be the points obtained by appending a 1 as the last coordinate of each point in $\cP$. 
Recall that $\cP$ is affinely independent if and only if $\hat \cP$ is linearly independent. 
Now, notice that a vector $w \in \R^{d+1}$ such that $w\cdot \hat p > 0$ for all $\hat p \in \hat \cP_1$, $w\cdot \hat p < 0$ for all $\hat p \in \hat \cP_2$ corresponds to an affine hyperplane separating $\cP_1$ from $\cP_2$. 

Finally, we prove $(b)$. Notice that a vector $p \in \R^{d}$
such that $p \cdot h > 0$ for all $h\in \cH_1$, $p \cdot  h < 0$ for all $h \in \cH_2$ corresponds to a point which is on the positive side of all hyperplanes in $\cH_1$ and on the negative side of all hyperplanes in $\cH_2$. 

%
%
\end{proof}

An alternate way to state these results is that a set of $m$ points is affinely independent if and only if it has $2^m$ topes, i.e. each of the $2^m$ elements of its power set can be separated from its complement by a hyperplane. 
Likewise, a set of $n$ hyperplanes has linearly independent normal vectors if and only if it has $2^n$ topes, i.e.  it splits space into $2^n$ full-dimensional chambers. 
Thus we can determine the dimension of a point configuration or hyperplane arrangement from the topes.
%
This is formalized by the \emph{VC dimension}, introduced in \cite{vapnik1971uniform}. 
While the VC dimension is typically defined for a family of sets or of binary vectors, we change the notation to consider families of sign vectors with no zero entries.


\begin{defn}\label{def:vc_sign}
Let $\Sigma \subseteq\{+, -\}^n$ be a set of sign vectors on $[n]$ with no zero entries. For $\tau \subseteq [n]$, define the restriction $\Sigma|_\tau$ to be the restriction of each sign vector in $\Sigma$ to indices in $\tau$.  We say that $\tau$ is \emph{shattered} by $\Sigma$ if  $\Sigma|_{\tau}$ contains  $2^{|\tau|}$ unique elements.  The \emph{VC dimension} of $\Sigma$, written $\dim_{\VC}(\Sigma)$ is the maximum size of a set $\tau$ shattered by $\Sigma$.  
\end{defn}

\begin{lem}\label{lem:vc_point}
The VC dimension of the set of topes $\mathcal T(\mathcal P)$ of a point arrangement with affine span $\R^{d}$ is $d+1$.
\end{lem}

\begin{proof}
Let $\cP$ be a point arrangement in $\R^{d}$ which has affine span $\R^{d}$. 
Then there is a set of $d + 1$ points $\cP' \subseteq \cP$ which are affinely independent.  
Then by Radon's theorem, for each $\cP_1, \cP_2 \subseteq \cP'$ there is an affine hyperplane separating the points of $\cP_1, \cP_2$. 
Thus, all $2^{d+1}$ choices of $\cP_1$ and $\cP_2$ give a unique member of $\cT(\cP)|_{\cP'}$. 
This means that the VC dimension of $\cT(\cP)$ is at least $d+1$. 
Since any collection of more than $d+1$ points is \emph{not} affinely independent, Radon's theorem guarantees that at least one partition $\cP_1, \cP_2 \subseteq \cP'$ which cannot be separated, thus the VC dimension is exactly $d+1$. 
\end{proof}

\begin{lem}\label{lem:vc_plane}
The topes $\mathcal T(\mathcal H)$ of a hyperplane arrangement whose normal vectors span $\R^{d}$ have VC dimension $d$. 
\end{lem}

\begin{proof}
Let $\cH$ be a hyperplane arrangement whose normal vectors spans $\R^{d}$. 
Then there must be a set $\cH'$ of $d$ hyperplanes whose normal vectors are linearly independent. 
Then by Radon's theorem, for each disjoint $\cH_1, \cH_2 \subseteq \cH'$, there exists a point which is on the positive side of all hyperplanes of $\cH_1$ and on the negative side of all hyperplanes of $\cH_2$. 
Thus, all $2^{d}$ choices of $\cH_1$ and $\cH_2$ give unique elements of $\cT(\cH)|_{\cH'}$. 
Since any collection of more than $d+1$ points is \emph{not} affinely independent, Radon's theorem guarantees that at least one partition is missing, thus the VC dimension is exactly $d$. 
\end{proof}

Notice that the VC dimension of a point arrangement in $\R^d$ is $d+1$, while the VC dimension of a hyperplane arrangement is $d$, due to the difference between affine and linear independence. 
In both cases, we can use the VC dimension of the set of topes as a way to ``sense" the dimension of the point or hyperplane arrangement they arose from. 
We will now show it is also possible to do this using subsets of topes  from matrices. 

\subsection{Sign vectors from matrices}

In this section, we show that if $(\cP, \cH, \cF)$ is a rank $d$ realization of a matrix $A$, it is possible to recover subsets of $\cT(\cP)$ and $\cT(\cH)$ from $A$, and to compute a lower bound on $\monr(A)$  based on the VC dimension of these subsets. 

\subsubsection{Threshold topes}
As we sweep a hyperplane with normal vector $h$ past the points of $\cP$, stopping at any time induces a partition of the points.
If $(\cP, \cH, \cF)$ is a rank $d$ representation of $A$, then the set of partitions of the points of $\cP$ induced by hyperplanes whose normal vectors lie in  $\cH$ can be recovered from the order of entries of $A$. This is a subset of $\cT(\cP)$. 

\begin{defn}
Let $A$ be a $m\times n$ matrix. For each each column $j$ of $A$ and threshold $\theta \in \R \setminus \{a_{ij}\mid i\in [m]\}$, we obtain a sign vector 
\[\sigma_{j}(\theta) := (\sign(a_{1j} - \theta), \ldots, \sign (a_{mj} - \theta) ).\]
Define the set of \emph{threshold topes} of $A$ as the set 
\[ \Sigma_{\thresh}(A) = \{\pm\sigma_j(\theta) \mid j\in [n], \theta \in  \R \setminus \{a_{ij}\mid i\in [m]\}\} \]
\end{defn}

The threshold topes recover the partitions of $\cP$ induced by affine hyperplanes whose normal vectors lie in $\cH$, and are thus a subset of $\cT(\cP)$. 
\begin{lem}\label{lem:thresh}
If $A$ has a rank $d$ representation $(\cP, \cH, \cF)$, then
$\Sigma_{\thresh}(A) \subseteq \cT(\cP)$. 
\end{lem}

\begin{proof}
If $(\cP, \cH, \cF)$ is a rank $d$ representation of $A$, then $\sign (a_{ij} - \theta) = \sign(p_i \cdot h_j - f_j\inv(\theta))$. 
Let $H_{j}(\theta)$ be the hyperplane $\{v\mid v \cdot h_j -  f_j\inv(\theta) = 0\}$. 
Then  $\sign(p_i \cdot h_j - f_j\inv(\theta)) = +$ if $p_i$ is on the positive side of  $H_{j}(\theta)$ and   $\sign(p_i \cdot h_j - f_j\inv(\theta)) = -$ if $p_i$ is on the negative side of  $H_{j}(\theta)$. 
Thus $\sigma_j(\theta) \in \cT(\cP)$. 
\end{proof}

\begin{ex}\label{ex:threshold}
We consider again the matrix $A$ from Example \ref{ex:distortion}. 
We have 
\begin{dmath*}
\Sigma_{\thresh}(A) = \{+--+, -++-, +-+-, -+-+, ++++, ----, ++-+, --+-, +++-, ---+, +-++, -+--\}
\end{dmath*}
Notice that each element $\sigma \in \Sigma_{\thresh}(A)$ corresponds to an element of $\cT(\cP)$, i.e. a way to partition the points using a hyperplane. 
We illustrate this in Figure \ref{fig:threshold_vectors}.
\end{ex}

\begin{figure}[ht!]
\includegraphics[width = 4 in]{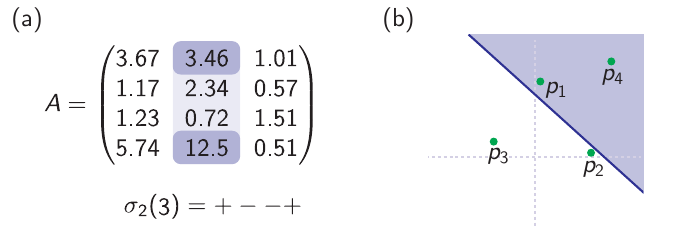}
\caption{\label{fig:threshold_vectors}
The threshold vectors $\Sigma_{\thresh}(A)$. (a) We have $\sigma_{2}(3) = +--+$, since only the first and fourth entries in the second column of $A$ are above the threshold 3.  (b) Only the points $p_1$ and $p_4$ are on the positive side of the illustrated hyperplane. The illustrated hyperplane has normal vector $h_2$. 
}
\end{figure} 

If the set of normal vectors is large enough, then the threshold topes recover $\cT(\cP)$.
Since the VC dimension of $\cT(\cP)$ can be used to recover the dimension of $\cP$, we can also use  $\Sigma_{\thresh}(A)$ to give a lower bound for the monotone rank of $A$. 

\begin{prop}\label{prop:vc_rank_thresh}
For any matrix $A$, 
\[ \dim_{\VC}(\Sigma_{\thresh}(A)) -1 \leq \monr(A).\]
\end{prop}

\begin{proof}
Suppose $A$ has monotone rank $d$. 
Then $A$ has a rank $d$ realization $(\cP, \cH, \cF)$. 
By Lemma \ref{lem:thresh}, 
$\Sigma_{\thresh}(A) \subseteq \cT(\cP)$. 
Thus, $$ \dim_{\VC}(\Sigma_{\thresh}(A)) \leq \dim_{\VC}(\cT(\cP)).$$
Now, by Lemma \ref{lem:vc_point}, $\dim_{\VC}(\cT(\cP)) = d + 1$. 
Thus \[\dim_{\VC}(\Sigma_{\thresh}(A)) \leq d+1.\] 

\end{proof}

Motivated by this, we define the Radon rank of a matrix as a lower bound for underlying rank. The Radon rank is also defined in \cite{curto22novel}. 

\begin{defn}
Define the \emph{Radon rank} of a matrix $A$ as 
\[\radr(A) =  \dim_{\VC}(\Sigma_{\thresh}(A)) -1\]
\end{defn}

\noindent Notice that $\radr(A) \leq \mrank(A)$.
Thus, we have proved the first part of Theorem \ref{ithm:rad_vc}.  

The set of threshold topes $\Sigma_{\thresh}(A)$ does not determine $\Pi(A)$, since the set of topes $\cT(\cP)$ of a point arrangement does not fully specify the sweep permutations $\Pi(\cP)$. 
For instance, in Figure \ref{fig:diff_more_than_thresh} we give an example of two point configurations which have the same topes, but different sets of sweep permutations. 
Motivated by this, we will define a different set of sign vectors which does retain all the information of $\Pi(A)$. 

\begin{figure}[ht!]
\includegraphics[width = 6 in]{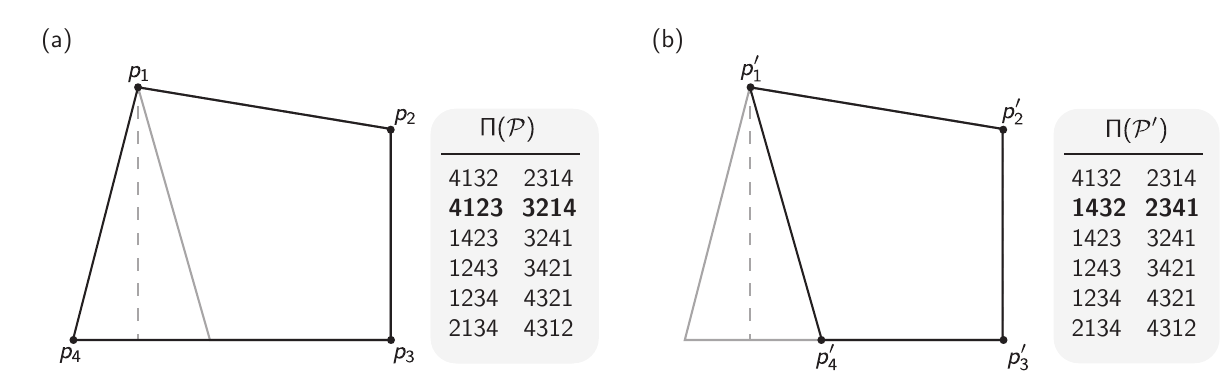}
\caption{\label{fig:diff_more_than_thresh}
Two point arrangements with the same set of topes, $\{+, -\}^4 \setminus \{+-+-, -+-+\}$, but differing sets of sweep permutations. (a) 4132 and 3214 are permutations of $\Pi(\cP)$, but not $\Pi(\cP')$. (b) 1432 and 2341 are permutations of $\Pi(\cP')$, but not $\Pi(\cP)$. 
As we slide  the point $p_4$ along the line connecting it to $p_3$, the change in  $\Pi(\cP)$ happens when $p_4$ crosses the dashed line, and the line segment connecting $p_1$ to $p_4$ becomes parallel to the line segment connecting $p_2$ to $p_3$. 
}
\end{figure}

\subsubsection{Difference topes}
If $f_j$ is a monotone function, then $a_{ij} > a_{kj}$ if and only if $f_j(a_{ij}) > f_j(a_{kj})$. 
Equivalently, 
$\sign(a_{ij} - a_{kj}) = \sign(f_j(a_{ij}) -  f_j(a_{kj}))$.
 We can use this fact to recover a subset of $\cT(\cH)$ from $A$ whenever $(\cP, \cH, \cF)$ is a rank $d$ realization of $A$. 

\begin{defn} For $i\neq k \in [n]$, we define the sign vector \[\sigma_{ik} := (\sign(a_{i1} -  a_{k1}), \ldots, \sign(a_{in} -  a_{kn})).\]
Define the set of \emph{difference topes} of $A$ as 
\[\Sigma_{\diff}(A) := \{\sigma_{ik} \mid {i, k}\in [m]\}\]
\end{defn}

This set of sign vectors appears in \cite{dunn2018signed} as the set of observable sign vectors. 
\begin{ex}
Returning to the matrix $A$ from Example \ref{ex:distortion}, we have difference topes 
\begin{align*}
\Sigma_{\diff}  = \{+++, ++-, --+, ---, -+-, +-+\}.
\end{align*}
\end{ex}

The difference topes lend themselves to a different geometric interpretation of $\cP$ and $\cH$, with elements of $\cH = \{h_1, \ldots, h_n\}$ as normal vectors of a central hyperplane arrangement $\cH = \{H_1, \ldots, H_n\}$, and the elements of $\cP - \cP = \{p_i - p_k \mid p_i, p_k \in \cP\}$ as points.  
In this view, the difference topes recover a subset of the topes of $\cH$. This is illustrated in Figure \ref{fig:hyperplanes_diff}. 
\begin{lem}\label{lem:diff}
If $A$ has a rank $d$ representation $(\cP, \cH, \cF)$, then
$\Sigma_{\diff}(A) \subseteq \cT(\cH)$. 
\end{lem}

\begin{proof}

If $(\cP, \cH, F)$ is a rank $d$ representation of $A$, then 
\begin{align*}
a_{ij} - a_{kj} &= f_j(p_i \cdot h_j) - f_j(p_k \cdot h_j)\\
\sign(a_{ij} - a_{kj}) &= \sign(p_i \cdot h_j - p_k \cdot h_j) = \sign((p_i-p_k) \cdot h_j).
\end{align*}
Thus
\[\sigma_{ik}= (\sign((p_i - p_k)\cdot h_1), \ldots, \sign((p_i - p_k)\cdot h_n)).\]
Then $\sign((p_i - p_k) \cdot h_j) = +$ if the point $(p_i - p_k)$ is on the negative side of the hyperplane $H_j$, and $\sign((p_i - p_k) \cdot h_j) = -$ if the point $(p_i - p_k)$ is on the negative side of $H_j$. 
Thus $\sigma_{ik}\in \cT(\cH)$.

\end{proof}

\begin{figure}[ht!]
\includegraphics[]{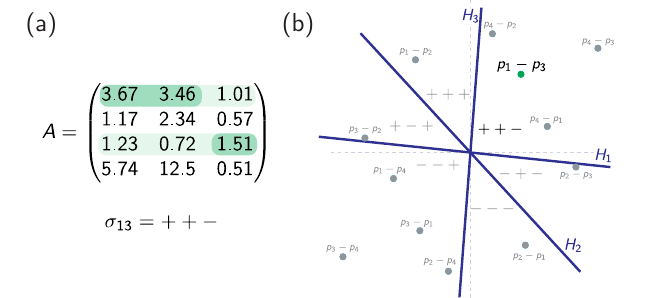}
\caption{\label{fig:hyperplanes_diff}
The difference vectors $\Sigma_{\diff}(A)$. 
(a) We have $\sigma_{13} = ++-$, since $a_{11} - a_{31} > 0, a_{12} - a_{32} >0, a_{13} - a_{33} < 0$. (b) The point $p_1 - p_3$ lies in the chamber $++-$,  on the positive sides of $H_1$ and $H_2$ and the negative side of $H_3$. 
}
\end{figure}

The set of difference topes $\Sigma_{\diff}(A)$ is sufficient to recover $\Pi(A)$. To see this, note that the $j^{th}$ entry of $\sigma_{ik}$ records whether $i$ or $k$ comes first in the $j^{th}$ column permutation. 
Again if the set of normal vectors is large enough, $\Sigma_{\diff}(A)$ is sufficient to recover the monotone rank of $A$ via the VC dimension of $\cT(\cH)$. Our goal is to recover the dimension of $\cH$ from the partial information provided by $\Sigma_{\diff}(A)$.

We can use the VC dimension to give lower bounds for monotone rank based $\Sigma_{\diff}(A)$. 

\begin{prop}\label{prop:vc_rank_diff}
For any matrix $A$, 

\[ \dim_{\VC}(\Sigma_{\diff}(A)) \leq \monr(A).\]
\end{prop}

\begin{proof}
Suppose $A$ has monotone rank $d$. 
By Lemma \ref{lem:diff}, 
$$\Sigma_{\diff}(A) \subseteq \cT(\cH).$$
Thus, 
$ \dim_{\VC}(\Sigma_{\diff}(A)) \leq \dim_{\VC}(\cT(\cH))$. 
Now, by Lemma \ref{lem:vc_plane}, $\dim_{\VC}(\cT(\cH)) = d$. 
Thus \[\dim_{\VC}(\Sigma_{\diff}(A)) \leq d.\] 
\end{proof}

\begin{defn}
Define the \emph{VC rank} of a matrix $A$  as 
\[\vcrank(A) = \dim_{\VC}(\Sigma_{\diff}(A)) \]
\end{defn}

Combining Propositions \ref{prop:vc_rank_thresh} and \ref{prop:vc_rank_diff}, we obtain a proof of Theorem \ref{ithm:rad_vc} from the introduction.

\begin{ithm}
For any matrix $A$, the Radon rank and the VC rank are both lower bounds for the monotone rank. 
\end{ithm}

The lower bounds provided by $\radr(A)$ and  $\vcrank(A)$ give a partial solution to the monotone rank estimation problem:
in the limit where a large amount of data is available relative to the monotone rank, these bounds are sufficient to recover $\mrank(A)$.
However, the size of a matrix $A$ needed to have either  $\radr(A) = d$ and  $\vcrank(A) = d$ grows exponentially in $d$ \cite{curto22novel}. 
We will show that monotone rank can grow faster than this by applying results about sign rank. 
We will then develop the connection between monotone rank and oriented matroid theory which will allow us to give lower bounds on monotone rank from less data, and to show that computing monotone rank exactly is hard.

\subsection{Sign rank}
The \emph{sign rank} of a matrix, used in theoretical computer science, is the minimal rank consistent with its sign pattern
\cites{alon2014sign, forster2002linear, basri2009visibility}. 

\begin{defn}
The \emph{sign rank} $\rank_{\sign}(A)$  of a matrix $A$ is the smallest rank $d$ such that there exists a rank $d$ matrix $B$ such that $\sign(A_{ij}) = \sign(B_{ij})$.
\end{defn}

We can relate the monotone rank of a matrix to the sign ranks of  $\Sigma_{\thresh}(A)$  and $\Sigma_{\diff}(A)$, each viewed a matrix. 
We will let  $\Sigma_{\thresh}(A)$ refer to the matrix whose columns are elements of  $\Sigma_{\thresh}(A)$ and let  $\Sigma_{\diff}(A)$ refer to the matrix whose rows are elements of $\Sigma_{\diff}(A)$.

\begin{prop}\label{prop:srmr}
For any matrix $A$, 
\[\rank_{sign}(\Sigma_{\thresh}(A)) \leq \monr(A) + 1\]
and
\[\rank_{sign}(\Sigma_{\diff}(A)) \leq \monr(A).\]
\end{prop}

\begin{proof}
Suppose $A$ has monotone rank $d$. 
Then let $B$ be a rank $d$ matrix with $B_{ij} = f_j(A_{ij})$. 
In particular, its column rank is $d$, thus there is a set of $d$ column vectors $b_1, \ldots, b_d$ such that any column $b$ of $B$ can be obtained as a linear combination $b = \lambda_1 b_1 + \ldots + \lambda_d b_d$ of these vectors. 
Now, recall that each column of $\Sigma_{\thresh}(A)$ is of the form  
\[\sigma_{j}(\theta) = (\sign(a_{1j} - \theta), \ldots, \sign (a_{mj} - \theta) ) = (\sign(b_{1j} - f\inv(\theta)), \ldots, \sign (b_{mj} - f\inv(\theta))) .\]
Thus, each column $\sigma_{j}(\theta)$ of $\Sigma_{\thresh}(A)$ has the same sign pattern of a vector of the form 
$b = \lambda_1 b_1 + \ldots + \lambda_d b_d  + \lambda_{d+1} \mathbf 1$, where $\mathbf 1$ is the all ones vector. 
Thus, the sign rank of  $\Sigma_{\thresh}(A)$ is at most $d +1 $. 

Likewise, the row rank of $B$ is $d$, so there is a set of $d$ row vectors such that every  row $b$ of $B$ can be obtained as a linear combination $b = \lambda_1 b_1 + \ldots + \lambda_d b_d$ of these vectors. 
Now, recall that each row of $\Sigma_{\diff}(A)$ is of the form \[\sigma_{ik} = (\sign(a_{i1} -  a_{k1}), \ldots, \sign(a_{in} -  a_{kn})) = 
 (\sign(b_{i1} -  b_{k1}), \ldots, \sign(b_{in} -  b_{kn})).\]
 Thus each row of $\Sigma_{\diff}(A)$ has the same sign pattern as a vector of the form $b = \lambda_1 b_1 + \ldots + \lambda_d b_d$, so the sign rank of  $\Sigma_{\diff}(A)$ is at most $d$. 
\end{proof}

We can use this relationship between monotone rank and sign rank to show that the monotone rank of a matrix can be higher than that provided by the VC rank or Radon rank. 

We will use a result of \cite{forster2002linear} which bounds the sign rank of a sign matrix $M$ in terms of $||M||$. Recall that for any matrix $M$, $||M||$ is the greatest singular value of $M$.

\begin{thm*}[\cite{forster2002linear}]
Let $M\in \{\pm 1\}^{m\times n}$ be a sign matrix of sign rank $d$. Then $$d\geq \frac{mn}{||M||}.$$ 
\end{thm*}

Using this together with Proposition \ref{prop:srmr}, we can bound the monotone rank of $A_\Sigma$ in terms of the singular values of the sign matrix $\Sigma$. 

The cited theorem, and related results, make it possible to construct explicit examples of matrices with high sign rank relative to their size. 
The primary example given in \cite{forster2002linear} is the family of Hadamard matrices $H_n$. These are a family of symmetric matrices with entries $\pm 1$ whose rows are pairwise orthogonal. They are defined recursively, with $$H_0= (1),$$ and $$H_{n+1} = \begin{pmatrix}
H_n & H_n \\ H_n & -H_{n}
\end{pmatrix}.$$ 
Notice that $H_{n}$ is $2^n\times 2^n$. Let $N:=2^n$, so that $H_n$ is $N\times N$.  
Now, we find $||H_n||$.
Now, because the columns of $H_n$ are pairwise orthogonal and the entries are $\pm 1$, we note that 
$$H_n^T H_n = NI.$$
This has eigenvalues $N, \ldots, N$. Thus, $||H_n|| = \sqrt{N}$. Thus, we have 
$$\sign\rank (H_n )\geq \frac{N}{\sqrt N} = \sqrt{N}.$$

We can  use the Hadamard matrices to construct examples of matrices with high monotone rank.

\begin{ithm}\label{thm:hadamard}
For each $N = 2^n$,  there exists a  $N\times N$ order matrix $A_N$ with $$\monr(A_{N}) \geq  \sqrt{N}-1.$$
\end{ithm} 

\begin{proof}

We first show that we can encode any set $n$ sign vectors of length $m$ as a subset of the threshold topes of a $m\times n$ matrix. 
Let $\sigma$ be a sign vector. 
Define a permutation $\pi(\sigma)$ of $[m]$ by setting $i$  before $j$ if $\sigma_i = -$ and $\sigma_j = +$ or $\sigma_i = \sigma_j$ and $i < j$. 
Now,  define $A_{\Sigma}$ to be a matrix with column orders $\pi(\sigma)$ for each $\sigma \in \Sigma$. 
Notice that \[\Sigma \subseteq \Sigma_{\thresh}(A_\Sigma).\]

Now, we apply this to the Hadamard matrices, setting $A_N  = A_{H_N}$. Then applying Proposition \ref{prop:srmr} and the result of \cite{forster2002linear}, we have 
\[\sqrt{N} \leq \rank_{\sign}(H_N) \leq \rank_{\sign}(\Sigma_{\thresh}(A_{H_N})) \leq \monr(A_{H_n}) + 1\]
thus 
\[\sqrt{N} -1 \leq \monr(A_{H_n})\]
as desired. 
\end{proof}

Note that this result appears as a bound for the underyling rank, rather than the monotone rank, in \cite{curto22novel}. 

This means that for large $N$, the matrices $A_{N}$ have monotone rank exceeding their Radon rank. 
Explicitly, we can check that $\pm H_n$, viewed as a set of sign vectors, has VC dimension $n$. 
Thus, for $n \geq 5$, $\sqrt{2^n} > n$, thus sign rank of $H_n$ exceeds the bound from VC dimension. 

In general, sign rank is difficult to compute, as we will discuss in Section \ref{sec:complexity}. 

\section{Oriented matroids and monotone rank}
Through the connection to sign rank, we have shown that monotone rank can exceed the bounds provided by VC rank and Radon rank. However, this does not provide us with small, interpretable examples of matrices for which monotone rank exceeds Radon rank.
In particular, the first case of a Hadamard matrix $H_n$ for which sign rank bound $\sqrt{2^n}$ exceeds the VC dimension $n$ is $H_5$, which is a $32 \times 32$ matrix. 
In this section, we use \emph{oriented matroid theory} to construct smaller examples of matrices whose sign rank or whose monotone rank exceeds the bounds which come from the VC dimension. 
We begin with an example, before introducing the general theory in Subsection \ref{sec:matroids} and rephrasing the monotone rank problem as an example of \emph{oriented matroid completion} in Subsection \ref{sec:matroid_completion}.

\begin{ex}\label{ex:rad_strict}
The first inequality in Proposition \ref{prop:vc_rank_thresh} may be strict. In particular, the order matrix $$ A = \begin{pmatrix}
12 & 13 & 3& 10 & 6 \\
13 & 14 & 4 & 9 & 5 \\
3 & 4 & 15 & 11 & 1 \\
10 & 9 & 11 & 8 & 2 \\
6 & 5 & 1 & 2 & 7 \\
\end{pmatrix}
$$
has Radon rank two and monotone rank three. 
\end{ex}

\begin{proof}

To see this, we first compute $\radr(A)$.  Since the largest sets of points we can shatter have size $3$, $\radr(A) = 2$. Now, we show that the monotone rank is, in fact, 3. 

We suppose for the sake of contradiction that $\monr (A) = 2$. Let $(\cP, \cH, \cF)$ be a rank 2 realization of $A$. 
By Radon's theorem, every subset of $\{p_1,p_2, p_3, p_4,p_5\}$ of size 4 must have a Radon partition. 
This means that each way we restrict the topes $\cT(\cP)$ to a set $\rho \subseteq [5]$ of size 4, there must be at least one ``potential Radon partition", i.e. an element of  $\{+, -\}^{\rho} \setminus (\cT(\cP)|_{\rho})$.  
In particular, since $\Sigma_{\thresh}(A) \subseteq \cT(\cP)$, we say that a potential Radon partition of $\rho$ is an element of $$\{+, -\}^{\rho} \setminus (\Sigma_{\thresh}(A)|_{\rho}).$$

For each subset of size 4, we compute the set of potential Radon partitions, writing these sign vectors as signed sets. We find that there is exactly one potential Radon partition for each subset:
\begin{align*}
1234: (14, 23)\\
1235: (1, 235)\\
1425: (14, 25)\\
1345: (135, 4)\\
2345: (235, 4)
\end{align*} 
Now, suppose $V$ is an arrangement of points with these Radon partitions, as illustrated in Figure \ref{fig:non_radon}. 
Then the partition $(1, 235)$ implies that $p_1 \in \conv(p_2, p_3, p_5)$. 
The partition $(14, 25)$ implies that the line segment from $p_1$ to $p_4$ crosses out of the triangle $\conv(p_2, p_3, p_5)$ by crossing the line segment from $p_2$ to $p_5$. However, the partition $(14, 23)$ implies that the line segment from $p_1$ to $p_4$ crosses out of the triangle $\conv(p_2, p_3, p_5)$ by crossing the line segment from $p_2$ to $p_3$.  
Thus, we have reached a contradiction. 
\end{proof}
\
\begin{figure}[ht!]
\includegraphics[width = 2 in]{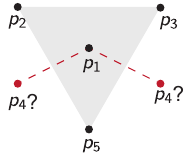}
\caption{\label{fig:non_radon}
There is no way to place the points $p_1, p_2, p_3,$ and $p_4$ consistent with the Radon partitions allowed by $A$.  
}
\end{figure}

\subsection{Oriented matroids}
\label{sec:matroids}
In this subsection, we expand upon this example using oriented matroid theory, which abstracts and generalizes the properties of both point arrangements and hyperplane arrangements. 
This will allow us to construct more examples.\
We begin by defining the oriented matroids of point and hyperplane arrangements. A more thorough introduction to oriented matroids can be found in \cite{bjorner1999oriented}. 

\begin{defn}
Let $\cP\subset \R^d$ be a point arrangement. Let $H$ be an affine hyperplane in $\R^d$. The \emph{covectors} of $\cP$, written $\cV^*(\cP)$ are the sign vectors which arise as 
\begin{align*}
(\sigma_H)_i = \begin{cases} 
+ & p_i \in H^+\\
- & p_i \in H^-\\
0 & p_i \in H
\end{cases}
\end{align*}

Let $\cH\subset \R^d$ be a central  hyperplane arrangement. Let $v$ be a point $\R^d$. The \emph{covectors} of $\cH$, written $\cV^*(\cH)$ are the sign vectors which arise as 
\begin{align*}
(\sigma_v)_i = \begin{cases} 
+ & v \in H_i^+\\
- & v \in H_i^-\\
0 & v \in H_i
\end{cases}
\end{align*}
\end{defn}

We define a partial order on sign vectors by inclusion, with $X \leq Y  \iff X^+ \subseteq Y ^+, X^- \subseteq Y^-$. 
Under this ordering, notice  that the topes $\cT(\cP)$ and $\cT(\cH)$ are the maximal elements of $\cV^*(\cP)$ and $\cV^*(\cH)$. 

The sets $\cV^*(\cH)$ and $\cV^*(\cP)$ follow the \emph{covector axioms for oriented matroids}.
In order to state them, we introduce some more notation. 
The \emph{support} of a sign vector $X$ is the set
$\underline X := \{ i \mid X_i \neq 0\}$. 
The \emph{positive part} of $X$ is $X^+ := \{i \mid X_i = +\}$ and the
\emph{negative part} is $X^- := \{i \mid X_i = -\}$. 
The \emph{composition} of sign vectors $X$ and $Y$ is defined component-wise by
\begin{align*}
  (X\circ Y)_i := 
  \begin{cases}
    X_i \mbox{ if } X_i\neq 0\\
    Y_i \mbox{ otherwise}.
  \end{cases}
\end{align*}
The \emph{separator} of $X$ and $Y$ is the unsigned set $\sep(X, Y) := \{ i\mid X_i = -Y_i\neq 0\}$. 
\begin{defn}
  Let $E$ be a finite set, and $\cV^* \subseteq 2^{\pm E}$ a collection of sign vectors satisfying the following
  \emph{covector axioms}:
\begin{enumerate}[(V1)]
\item\label{axiom:emptysetV} $\emptyset \in \mathcal V^*$
\item\label{axiom:symmetryV} $X\in \mathcal V^* $ implies $-X\in \mathcal V$. 
\item\label{axiom:compositionV} $X, Y\in \mathcal V^*$ implies $X\circ Y\in \mathcal V^*$. 
\item\label{axiom:crossingV} If $X, Y\in \mathcal V^*$ and $e \in \sep(X,Y)$, then there exists $Z\in \mathcal V^*$ such that $Z_e = 0$ and $Z_f = (X\circ Y)_f = (Y\circ X)_f$ for all $f\notin \sep(X, Y)$. 
\end{enumerate}
Then, the pair $\cM = (E, \cV^*)$ is called an \emph{oriented matroid}, and $\cV^*$ its set of covectors.
\end{defn}

While every point or hyperplane arrangement gives rise to an oriented matroid, not all oriented matroids arise this way. 
An oriented matroid is \emph{representable} if and only if it arises from a hyperplane arrangement (or, equivalently, if and only if it arises from a point arrangement). 
However, oriented matroids capture a good deal of information about an arrangement. 
In particular, we can define the \emph{rank} of an oriented matroid as the VC dimension of its set of topes. 
Under this definition, a point arrangement in $\R^{d+1}$ or a central hyperplane arrangement in $\R^{d}$ give rise to an oriented matroid of rank $d$.

There are a number of equivalent axiom systems for oriented matroids. 
We mention one dual description, the \emph{circuit axioms}, which we will use to shed light on Example \ref{ex:rad_strict}. We first describe the circuits and vectors in the representable case, which correspond to minimal Radon partitions.

\begin{defn}\label{def:circuits_rep}
Let $\cP\subset \R^d$ be an arrangement of $m$ points. The \emph{vectors} $\cV(\cP)$ of $\cP$ correspond to the Radon partitions of the points: a sign vector $X \in 2^{\pm{\cP}}$ is a vector if 
\[\conv(X^+) \cap \conv( X^-) \neq \emptyset.\]

The \emph{circuits} of $\cP$, written $\cC(\cP)$ are minimal vectors, corresponding to the minimal Radon partitions. 
\end{defn}

The vectors of an oriented matroid $\cM$ are the covectors of a dual oriented matroid $\cM^*$, and thus follow the covector axioms.
 The circuits of a point arrangement follow a list of rules know as the circuit axioms for oriented matroids. 
 Oriented matroids can also be defined via these axioms. 

\begin{defn}\label{D:circuitaxioms}
  Let $E$ be a finite set, and $\cC\subseteq 2^{\pm E}$ a collection of signed subsets satisfying the following \emph{circuit axioms}: 
\begin{enumerate}[(C1)]
	\item\label{axiom:emptysetC} $\varnothing \notin \cC$.
	\item\label{axiom:symmetryC} $X \in \cC$ implies $-X \in \cC$.
	\item\label{axiom:incomparableC} $X,Y \in \cC$ and $\underline X \subseteq \underline Y$ implies $X = Y$ or $X = -Y$.
	\item\label{axiom:weakelimC} For all $X,Y \in \cC$ with $X \neq -Y$ and an element $e \in X^+ \cap Y^-$, there is a $Z \in \cC$ such that $Z^+ \subseteq (X^+ \cup Y^+) \setminus e$ and $Z^- \subseteq (X^- \cup Y^-) \setminus e$.
\end{enumerate}
Then the pair $\cM = (E, \cC)$ is an oriented matroid, and $\cC$ is its set of circuits. 
\end{defn}

By Radon's theorem, if $\cP$ is a point arrangement in $\R^{d}$, any set of size $ m \geq d+2$ contains a Radon partition. 
This allows us to read the dimensionality of a point configuration off of its circuits. 
This is used to define the rank of an oriented matroid from of its set of circuits. 
\begin{defn}
An oriented matroid $\cM = (E, \cC)$ has \emph{rank} $d$ if and only if for each set $X \subseteq E$ of size at least $d + 1$, there exists a circuit $Y$ such that $\supp{Y} \subseteq \supp{X}$. 
\end{defn}

Generic point and hyperplane arrangements correspond to \emph{uniform} oriented matroids. 

\begin{defn}
An oriented matroid of rank $d$ on ground set $E$  is \emph{uniform} if each set $\sigma\subseteq E$ of size $d+1$ is the support of a circuit. 
\end{defn}

We observe a relationship between the covectors and vectors  of a point arrangement. 
Namely, if $X\in \cV(\cP)$ is a vector of $\cP$, then the points $X^+$, $X^-$ have intersecting convex hulls and thus cannot be separated by a hyperplane. 
Thus, there is no covector $Y\in \cV^*(\cP)$  containing $X$. 
We can extend this with \emph{orthogonality}. 
Two signed sets are orthogonal if they are neither equal nor opposite on their common support: $X$ and $Y$ are orthogonal, written $X\perp Y$, if either $\supp{X} \cap \supp{Y} = \emptyset$ or if there exist $e, f$ such that $X_e = Y_e$ and $X_f = - Y_f$.

Every covector of an oriented matroid is orthogonal to each vector, and vice versa. 
This is sufficient to characterize either the covectors or the vectors of an oriented matroid, given the other:

\begin{itemize}
\item $X$ is a vector of $\cM$ if and only if $X \perp Y$ for each covector $Y$ of $\cM$
\item $X$ is a covector of $\cM$ if and only if $X \perp Y$ for each vector $Y$ of $\cM$.
\end{itemize}

This will allow us to recover information about the circuits of a point arrangement or hyperplane arrangement using the information about topes we have from $\Sigma_{\thresh}(A)$ and $\Sigma_{\diff}(A)$. 
We can interpret Example \ref{ex:rad_strict} in terms of the circuit axioms for oriented matroids. 

\subsection{Oriented matroid completion}
\label{sec:matroid_completion}

In this section, we interpret the problem of computing monotone rank as an \emph{oriented matroid completion problem}. 
We reexamine Example \ref{ex:rad_strict} in light of this perspective.

\begin{prop}\label{prop:topes_completion}
If $A$ is a matrix of monotone rank $d$, then there exists a representable oriented matroid $\cM_{\thresh}$ of rank $d+1$ such that
\[\Sigma_{\thresh}(A) \subseteq \cT( \cM_{\thresh})\]
and a representable oriented matroid  $\cM_{\diff}$ of rank $d$ such that
\[\Sigma_{\diff}(A) \subseteq \cT( \cM_{\diff}).\]
\end{prop}

\begin{proof}
Suppose $A$ has monotone rank $d$. 
Let $(\cP, \cH, \cF)$ be a rank $d$ representation of $A$. 
By Lemma \ref{lem:thresh}, we have 
\[\Sigma_{\thresh}(A) \subseteq \cT(\cP)\]
and by Lemma \ref{lem:diff}, we have 
\[\Sigma_{\diff}(A) \subseteq \cT( \cH).\]
Letting $\cM_{\thresh}$ be the matroid of the point arrangement $\cP$ and letting $\cM_{\diff}$ be the matroid of  the point arrangement $\cP$, we obtain the desired result. 
\end{proof}

This motivates us to define the \emph{oriented matroid rank} of a set of sign vectors. We use this to define the oriented matroid completion rank of a matrix as a combinatorial relaxation of the monotone rank. 
To define this, we drop the condition that the oriented matroid be representable. 


\begin{defn}
The \emph{oriented matroid rank} of a set of sign vectors $\Sigma$, written $\orank(\Sigma)$ is the minimum rank of an oriented matroid $\cM$ such that $\Sigma \subseteq \cT(\cM)$.
\end{defn}

We use this to define the oriented matroid completion rank of a matrix. 

\begin{defn}
The oriented matroid completion rank of a matrix $A$ is
\[ \orank(A) := \max(\orank(\Sigma_{\diff}(A)), \orank(\Sigma_{\thresh}(A))-1)\]
\end{defn}

With this definition, we obtain Theorem \ref{ithm:omrank} from the introduction as a corollary of Proposition \ref{prop:topes_completion}. 
\begin{ithm}
For any matrix $A$, the oriented matroid completion rank is a lower bound for the monotone rank. 
\end{ithm}

\begin{proof}
Suppose $A$ is a matrix of monotone rank $d$. 
Then by Proposition \ref{prop:topes_completion}, there exist oriented matroids $\cM_{\thresh}$ and $\cM_{\diff}$  of rank $d+1$ and $d$, respectively such that \(\Sigma_{\thresh}(A) \subseteq \cT( \cM_{\thresh})\)
and
\(\Sigma_{\diff}(A) \subseteq \cT( \cM_{\diff}).\)
Thus $\orank(A) \leq d$.
\end{proof}

This gives us an interpretation of monotone rank estimation as an oriented matroid completion problem: given a set of sign vectors arising as either $\Sigma_{\thresh}(A)$ or $\Sigma_{\diff}(A)$, how can we find an oriented matroid of minimal rank whose topes contain these vectors? 
In essence, this acts as a combinatorial relaxation of the monotone rank estimation problem since the provided matroid may not be representable.
 
\begin{problem}\label{prob:oriented_matroid_completion}
Find the oriented matroid completion rank of a set of sign vectors.
\end{problem}

We can re-interpret Example \ref{ex:rad_strict} by considering a dual version of Proposition \ref{prop:topes_completion}.
To state this version, we first define the orthogonal complement of a set of sign vectors.

\begin{defn}
Let $\Sigma \subseteq 2^{\pm[m]}$ be a set of sign vectors. Define 
\[\Sigma^\perp = \{ Y\in  2^{\pm[m]} \mid Y \perp X \mbox{ for all } X\in \Sigma\}\]
\end{defn}

\begin{prop} \label{prop:vector_completion}
If $A$ is a matrix of monotone rank $d$, then there exists a representable oriented matroid $\cM_{\thresh}$ of rank $d+1$ such that
\[\Sigma_{\thresh}(A)^\perp \supseteq \cV( \cM_{\thresh})\]
and a representable oriented matroid  $\cM_{\diff}$ of rank $d$ such that
\[\Sigma_{\diff}(A)^\perp \supseteq \cV( \cM_{\diff}).\]
\end{prop}

\begin{proof}
By Proposition \ref{prop:topes_completion}, there exist representable oriented matroids $\cM_{\thresh}$ and $\cM_{\diff}$ such that \(\Sigma_{\thresh}(A) \subseteq \cT( \cM_{\thresh})\) and \(\Sigma_{\diff}(A) \subseteq \cT( \cM_{\diff}).\) 
We first consider $\cM_{\thresh}$. 
Now, suppose $X\in \cV(\cM_{\thresh})$. 
Then for any $Y\in \cT(\cM_{\thresh})$, we have $Y\perp X$. Since  \(\Sigma_{\thresh}(A) \subseteq \cT( \cM_{\thresh})\), this implies that $Y\perp X$ for any $Y \in \Sigma_{\thresh}(A) $. 
Thus, $X \in \Sigma_{\thresh}(A)^\perp$. 
By the same argument, if $X\in  \cV( \cM_{\diff})$, then 
$X\in \Sigma_{\diff}(A)^\perp$. 
\end{proof}

\begin{problem}
 \label{prob:dual_oriented_matroid_completion}
Given a set of sign vectors $\Sigma \subseteq 2^{\pm[m]}$, what is the minimal rank of an oriented matroid $\cM$ such that 
\[\Sigma \supseteq \cV(\cM)?\]
\end{problem}

The rank of an oriented matroid is most easily read off of its set of circuits. 
Because of this, we define the \emph{potential circuits} of a set of sign vectors. 

\begin{defn}
Let $\Sigma$ be a set of sign vectors. 
The \emph{potential circuits of rank $d$}, written $(\Sigma^\perp)_d$ are the elements of $X \in \Sigma^\perp$ with $|\underline{X}| =  d+1$. 
\end{defn}

\begin{prop}\label{prop:potential} 

If $\Sigma$ has oriented matroid completion rank $d$, then there exists an oriented matroid $\cM$ such that 
\begin{align*}
\cC(\cM) &\subseteq (\Sigma^\perp)_{d} 
\end{align*}
\end{prop}

\begin{proof}
Let $\Sigma$ be a set of sign vectors with oriented matoid rank $d$. Then by definition, there exists an oriented matroid $\cM$ with $\Sigma \subseteq \cT(\cM)$, $\Sigma^\perp \subseteq \cV(\cM)$. 

Without loss of generality, we can take  $\cM$ to be uniform. (In the realizable case, this follows from the fact that we can perturb the hyperplanes realizing an oriented matroid to obtain a realization in general position which still has all of the original topes.  The general case is covered in (\cite{bjorner1999oriented}, Corrollary 7.7.9).)
Thus, each set  $X\subseteq [m]$ of size $d+1$ is the support of a circuit of $\cM$.
These circuits must be orthogonal to every tope of $\cM$, thus
\begin{align*}
\cC(\cM) &\subseteq (\Sigma^\perp)_{d}.
\end{align*}

\end{proof}

\begin{ex}
We consider again the matrix 
$$ A = \begin{pmatrix}
12 & 13 & 3& 10 & 6 \\
13 & 14 & 4 & 9 & 5 \\
3 & 4 & 15 & 11 & 1 \\
10 & 9 & 11 & 8 & 2 \\
6 & 5 & 1 & 2 & 7 \\
\end{pmatrix}.
$$
We have already checked that $A$ has Radon rank two, i.e. $ \dim_{\VC}(\Sigma_{\thresh}(A)) = 3$. Now, we show that $\omrank(\Sigma_{\thresh}(A)) = 4$, thus $\omrank(A) \geq 3$. 
Suppose to the contrary that   $\omrank(\Sigma_{\thresh}(A)) = 3$.
Then there exists a rank three matroid $\cM_{\thresh}$ such that 
$$\cC(\cM_{\thresh}) \subseteq (\Sigma_{\thresh}(A)^\perp)_3.$$

Now, we compute  
\begin{align*}
(\Sigma_{\thresh}(A)^\perp)_3
 = \{&+--+0, -++-0, +--0-, -++0+, +-0+-,\\ &-+0-+, +0+-+, -0-+-, 0++-+, 0--+-\}.
 \end{align*}
 Written out for each support, these are the ``potential Radon partitions" of Example \ref{ex:rad_strict}. We find that each set of size four supports exactly one opposite pair of potential circuits. 

Thus, if there is any oriented matroid $\cM_{\thresh}$ of rank three such that 
$\cC(\cM_{\thresh}) \subseteq (\Sigma_{\thresh}(A)^\perp)_3,$
then this oriented matroid must satisfy 
$$\cC(\cM_{\thresh}) = (\Sigma_{\thresh}(A)^\perp)_3.$$
However, $(\Sigma_{\thresh}(A)^\perp)_3$ does not satisfy the circuit axioms for oriented matroids: it violates C4.
 We apply axiom C4 to $X = +--0-$, $Y = -+0-+$, $e = 5$. Then there must exist $Z \in  \cC$ with 
$Z^+ \subseteq X^+ \cup Y^+ \setminus \{e\} = \{1, 2\}$, 
$Z^- \subseteq X^- \cup Y^- \setminus \{e\} = \{1, 2, 3, 4\}$. 
No such potential circuit is present: the potential circuits on support $\{1, 2, 3, 4\}$ are $+--+0$ and $-++-0$, which do not conform to this pattern. Thus, by Proposition \ref{prop:potential}, 
$$3 \leq \omrank(A) \leq \mrank(A).$$

\end{ex}

\section{Computational complexity and matrices of monotone rank 2} 
\label{sec:complexity}

In this section, we show that computing the monotone rank of a matrix, or even deciding whether or not a matrix has monotone rank two, is computationally difficult. 
However, we also show that we can solve a combinatorial relaxation of this problem in polynomial time. 

\subsection{Computational Complexity}

\emph{The existential theory of the reals}, written $\exists \R$, is the complexity class of decision problems of the form
$$\exists(x_1 \in \R)\cdots \exists(x_n\in \R)P(x_1, \ldots , x_n),$$
where P is a quantifier-free formula whose atomic formulas are polynomial equations and inequalities in the $x_i$ \cite{broglia1996lectures}. 
Problems which are $\exists \R$-complete are not believed to be computationally tractable. 
In particular, they must be NP-hard.

Many classic problems in computational geometry fall into $\exists\R$ \cite{schaefer2009complexity}. 
Many of these results follow from the fact that
determining whether an oriented matroid is representable is $\exists \R$-complete \cites{sturmfels1987decidability, shor1991stretchability, mnev1988universality}. 

\begin{thm*}[\cites{sturmfels1987decidability, shor1991stretchability, mnev1988universality}]
The problem of deciding whether a uniform oriented matroid of rank $3$ is realizable is $\exists \R$-complete, and therefore NP-hard. 
\end{thm*}

For instance, the fact that sign rank is $\exists \R$-complete is established in  \cite{basri2009visibility} and independently in \cite{bhangale2015complexity}.  More precisely, for $r\geq 3$, the decision problem of determining whether a matrix $A$ has sign rank $r$ is $\exists \R$-complete. 
This follows from the $\exists \R$-completeness of determining whether an oriented is representable. 

Note that the fact that computing sign rank is $\exists \R$-complete does not immediately establish that computing  monotone rank is $\exists \R$-complete. 
However, we will  now show that the main result of \cite{hoffmann2018universality} implies that computing  monotone rank is also an $\exists \R$-complete problem.

\begin{ithm}\label{thm:complete}
The problem of deciding whether a $m\times n$ matrix $A$ has monotone rank 2 is $\exists \R$-complete, and therefore NP-hard.  
\end{ithm}

In order to state the main result of  \cite{hoffmann2018universality} and prove Theorem \ref{thm:complete}, we introduce \emph{allowable sequences}, which characterize the set  $\Pi(\cP)$, where $\cP$ is a point arrangement in the plane. 

Let $\cP$ be an arrangement of points in the plane. 
We consider the set of sweep permutations $\Pi(\cP)$, ordered as we rotate the sweep direction $h$ around a circle. 
In particular, the sweep permutation $\pi_h$ changes at the non-generic sweep directions: values of $h$ which are perpendicular to one or more of the lines containing two or more points of $\cP$. 
At these values, the line swept perpendicular to $h$ hits these points at the same time. 
At these values, the permutation changes by reversing the order of all points which the line hits at the same time. 
Thus, the sweep permutations  $\Pi(\cP)$ can be put in an order in which they form an \emph{allowable sequence}. 

\begin{defn}
\label{def:allowable}
Define $-\pi$ to be the permutation in the reverse order, i.e. if $\pi = 1342, -\pi = 2413.$ 

An \emph{allowable sequence} is a circular sequence of permutations $ \Pi = \pi_1,\pi_2, \ldots, \pi_{2m}$ of $[n]$ such that: 
\begin{enumerate}
    \item $\pi \in \Pi \Leftrightarrow -\pi\in \Pi$ 
    \item $\pi_{i+1}$ is obtained from $\pi_i$ by reversing the order of one or more disjoint substrings of $\pi_i$
    \item The order of each pair $i, j$ is reversed exactly once between each $\pi_{k}, -\pi_{k}$
\end{enumerate}
\end{defn}

An allowable sequence is called \emph{simple} if each pair $\pi_i, \pi_{i+1}$ differ by reversing a pair of adjacent entries. 
We are primarily concerned with simple allowable sequences because they correspond to point arrangements in general position. Our definition of an allowable sequences differs slightly from the standard one in \cite{goodman1993allowable} in that we consider an allowable sequence as a circular sequence of permutations; the standard definition considers the sequence of permutations beginning at one permutation and ending at its reverse permutation. Both definitions encode the same information. 

An allowable sequence is \emph{realizable} if it actually arises from a point configuration. 
As with oriented matroids, it is difficult to determine whether an allowable sequence is realizable, even for simple sequences. 
\begin{thm*}[Theorem 6, \cite{hoffmann2018universality}]
The realizability of simple allowable sequences is $\exists \R$-complete.
\end{thm*}

To prove Theorem \ref{thm:complete}, we thus need to reduce the problem of checking whether an allowable sequence is realizable to the problem of checking whether a matrix has monotone rank two. 

\begin{proof}[Proof of Theorem \ref{thm:complete}]
Given a simple allowable sequence $\Pi$ on $m$ elements, we produce a matrix $A$ which has monotone rank 2 if and only if $\Pi$ is realizable. 
For each $\pi \in \Pi$, we produce a column vector $a_{\pi}$  of length $m$ whose entries are in the order given by $\pi$. 
Let $A(\Pi)$ be the matrix whose columns are $a_{\pi}$ for $\pi \in \Pi$. 
Now, we show that $A(\Pi)$ has monotone rank 2 if and only if $\Pi$ is representable. 

If $\Pi$ is representable, then we can choose points $\cP = \{p_1, \ldots, p_m\}$ representing $\Pi$, together with a set of normal vectors $\cH = \{h_\pi \mid \pi \in \Pi\}$  such that $h_\pi$ sweeps past the points in the order given by $\pi$. 
Then since each column order of $A(\Pi)$ corresponds to the order in which a hyperplane swept normal to $h_\pi$ sweeps past the points of $\cP$, there is a family of monotone functions $\cF$ such that $(\cP, \cH, \cF)$ is a rank two representation of $A(\Pi)$.

Next, suppose $A(\Pi)$ has monotone rank two, with a rank two representation $(\cP, \cH, \cF)$. 
We show that the allowable sequence of $\cP$ is $\Pi$. 
To do this, we note that we have already shown that every permutation in $\Pi$ arises from a sweep of $\cP$. 
Thus, we just need to show that there are no more sweep permutations of $\Pi$.
This follows because $\Pi$ is a simple allowable sequence. 
Thus, $\Pi$ has length $2 \binom m 2  = (m)(m-1)$, since every switch between permutations corresponds to swapping the order of one pair of points, and each pair of points is swapped exactly twice. 
Also notice that no other allowable sequence on $m$ points can be longer than a simple sequence. 
Thus, every sweep permutation of $\Pi$ must be contained in $\cP$. 
\end{proof}

%
%

\subsection{Matroid completion in rank two}
Through Theorem \ref{thm:complete},  we have shown that monotone rank estimation is hard, even for monotone rank two. 
On the other hand, we  will next show that its combinatorial relaxation, Problem \ref{prob:oriented_matroid_completion}, is easy for rank two. 
This follows from the fact that oriented matroids of rank two have a very simple structure.

\begin{thm}\label{thm:rank_2_completion}
Given a set of sign vectors  $\Sigma\subseteq \{+, -\}^n$ of size $m$, there is an $O(mn)$ algorithm to determine whether there is a rank two oriented matroid $\cM$ such that \[\Sigma \subseteq \cT(\cM).\]
\end{thm} 

Notice that this has Theorem \ref{ithm:easy} from the introduction as a corollary: since the difference topes of an $m\times n$ matrix are a set of $m(m-1)$ sign vectors of length $n$, applying Theorem \ref{thm:rank_2_completion} to $\Sigma_{\diff}(A)$ gives a $O(m^2 n)$ algorithm to determine whether a $m\times n$ matrix has difference oriented matroid completion rank 2. 
If $\orank_{\diff}(A) > 2$, we can conclude that $\mrank(A) > 2$ as well. 
Thus, for matrices of monotone rank two, this combinatorial relaxation is much easier than the original problem. 
For higher ranks, we do not yet know whether or not this is the case. 

We will reserve a full proof of Theorem \ref{thm:rank_2_completion} and a description of the algorithm involved for the appendix. 
Intuitively, this result depends on the very simple structure of rank two oriented matroids, which we describe here. 
Every rank two oriented matroid is realizable, and has a realization with a central hyperplane arrangement in the plane, as illustrated in Figure \ref{fig:rank_2}. 
This gives rank two matroids a certain structure.

\begin{figure}[ht!]
\includegraphics[width = 2.5 in]{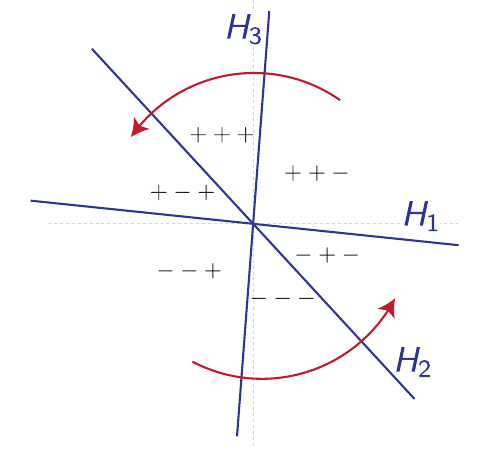}
\caption{
\label{fig:rank_2}
The topes of any rank two oriented matroid can be placed in a unique (up to reversal) circular order,  such that the sign of $X_i$ changes only once between $X$ and $-X$. This corresponds to traversing them counterclockwise in a circle centered at the origin. In the illustrated example, this order is $++-, +++, +-+, --+, ---, -+-$. 
}
\end{figure}

\begin{lem}\label{lem:rank_2_topes}
The topes of a rank two oriented matroid $\cM$ can placed in a unique (up to reversal) circular  order so that the sign of $X_i$ changes exactly once between $X$ and $-X$. 
Further, any set of sign vectors which is closed under negation and can be placed in such an order is the set of topes of a rank two oriented matroid. 
\end{lem}

\begin{proof}
Let $\cM$ be a rank two oriented matroid with topes $\cT(\cM)$. 
We show that $\cT(\cM)$ can be placed in a circular order such that for any $X$,  the sign of $X_i$ changes exactly once between $X$ and $-X$. 
By a basic result of oriented matroid theory \cite{bjorner1999oriented}*{Section 6.1}, every rank two matroid can be represented by a central hyperplane arrangement in $\R^{2}$. 
Now, order the topes based on the order in which a point traveling in a circle centered at the origin encounters the chamber corresponding to each tope. 
Notice that as we follow the circle,  we cross each hyperplane $H_i$ exactly once in each direction. 
Thus, between any $X, -X$, the sign of $X_i$ changes exactly once, so we have obtained the desired order. 

Now, suppose $\Sigma$ is a set of sign vectors which is closed under negation and can be placed in a circular order such that the sign of $X_i$ changes exactly once between each $X, -X$. 
We will use this order to construct a  central hyperplane arrangement in $\R^2$  with topes $\Sigma$. 
We divide the circle into $|\Sigma|$ equally sized sections, noting that $|\Sigma|$ is even since it is closed under negation. 
Thus, the lines defining these sections form a hyperplane arrangement. 

We show that this hyperplane arrangement is a realization of $\Sigma$ as a rank two oriented matroid. 
First, we notice that in this order, for each $Y\neq \pm X$, exactly one of $Y, -Y$ appears on the path from $X$ to $-X$. 
Thus, there are the same number of elements separating each $X, -X$, so the sections corresponding to $X$ and $-X$ are opposite one another in the arrangement. 
Thus, we can consistently label each hyperplane in the arrangement by the elements of $\sep(X, Y)$ for each pair $X, Y$ separated by it. 
Now, notice that because the sign of $X_i$ changes exactly once between $X, -X$, each $i$ appears in the label of only one hyperplane. 
For every hyperplane which is given multiple labels, we duplicate it into a family of overlapping hyperplanes, one with each label. 
We orient each hyperplane such that the topes with $X_{i} = +$ are on the positive side of $H_i$.  
Finally, we notice that this hyperplane arrangement realizes a rank two oriented matroid with topes $\Sigma$. 
\end{proof}

This structure makes topes of rank two oriented matroids easy to recognize, even when they are out of order or missing elements.
For instance, we consider the set of sign vectors $\Sigma = \{++++, ++--, -+-+, ----, --++, +-+-\}$. We see that, restricting to the first three components yields the sign pattern in Figure \ref{fig:rank_2}. Thus, up to reflection, the sign vectors must be placed in the unique order pictured in Figure \ref{fig:rank_2}, in order to satisfy Lemma \ref{lem:rank_2_topes} in the first three components. However, this ordering does not satisfy Lemma \ref{lem:rank_2_topes}  for the fourth component. Thus, this set of sign vectors has oriented matroid completion rank at least 3, even though it has VC dimension 2. 

The final step to prove Theorem \ref{thm:rank_2_completion} is to write down an algorithm which takes in a set of sign vectors and determines whether they can be put in an order which satisfies Lemma \ref{lem:rank_2_topes}. We do this in  Appendix \ref{sec:rank_2_proof}. 
Similar algorithms appear in slightly different contexts in \cite{basri2009visibility} and \cite{rosen2017convex}.

\section{Conclusion and Open Questions}
Via Theorem \ref{ithm:omrank}, we have shown that connecting monotone rank to oriented matroid theory can give us lower bounds which exceed those using only the VC dimension. 
Further, the fact that it is easy to identify matrices with oriented matroid completion rank two, even though it is already hard to identify matrices with monotone rank two, implies that this lower bound may be easier to compute in general. 
\label{sec:conclusion}
\begin{question}
Is there an efficient algorithm to solve Problem \ref{prob:oriented_matroid_completion} and its dual version, Problem \ref{prob:dual_oriented_matroid_completion} in general? 
\end{question}

In addition to giving an algorithm for oriented matroid completion problems, it is also interesting to consider what conditions a set of sign vectors must have in order to be completed to an oriented matroid of a specified rank. 

\begin{question}
Let $\Sigma\subseteq \{+, -\}^n$ be a set of sign vectors with VC dimension $d$.  Which additional conditions must $\Sigma$ satisfy in order to guarantee that there exists a rank $d$ oriented matroid $\cM$ such that 
$\Sigma \subseteq \cT(\cM)$? 
\end{question}

As a special case, we consider the question of whether every condition oriented matroid of rank $d$, as defined in \cite{bandelt2018coms} is contained within an oriented matroid. 
Further, we can also pose their Conjecture 1, that every conditional oriented matroid is a fiber of an oriented matroid, as an oriented matroid completion problem.

Our other open questions consider the relationship between monotone rank and oriented matroid rank, and on the ranges of possible values each of these quantities can take on matrices of a given size.
First, we ask how good of a lower bound oriented matroid completion rank is for monotone rank?

\begin{question}
How large can the gap between the oriented matroid completion rank of a matrix and its monotone rank be? 
\end{question}

Using the example of the Hadamard matrices, we have shown that monotone rank can grow with the square root of the matrix size. 
However, it is not known whether it is possible to exceed this bound. 

\begin{question}
How high can the monotone rank of a $m\times n$ matrix be? 
\end{question}

We can also ask this question for monotone rank. 

\begin{question}
How high can the oriented matroid completion rank of a $m\times n$ matrix be? 
\end{question}

In particular, while we have shown that the Hadamard matrices have high sign rank, we do not know their oriented matroid rank, and how it compares. 

\begin{question}
Let $\Sigma_N$ be the set of sign vectors arising as rows of the $N\times N$ Hadamard matrix, for each $N = 2^n$. What is the oriented matroid completion rank of $\Sigma_N$? 
\end{question}

\section*{Acknowledgments}
I would like to thank Carina Curto for initially proposing the problem of estimating underlying rank, and for giving me feedback, direction, and advice on this project. Additionally, I benefited from helpful discussions with Hannah Rocio Santa Cruz, Juliana Londono Alvarez, Vladimir Itskov, Laura Anderson, and Michael Dobbins. I was supported by the NSF fellowship DGE1255832 and the Center for Systems Neuroscience postdoctoral  fellowship at Boston University.

\bibliography{Biblio-Database}{}

\begin{bibdiv}
\begin{biblist}

\bib{ak2012opt}{article}{
      author={Akerboom, Jasper},
      author={Chen, Tsai-Wen},
      author={Wardill, Trevor~J},
      author={Tian, Lin},
      author={Marvin, Jonathan~S},
      author={Mutlu, Sevinc},
      author={Calderon, Nicole~Carreras},
      author={Esposti, Federico},
      author={others},
       title={Optimization of a gcamp calcium indicator for neural activity
  imaging},
        date={2012},
     journal={Journal of neuroscience},
      volume={32},
      number={40},
       pages={13819\ndash 13840},
}

\bib{alon2014sign}{inproceedings}{
      author={Alon, Noga},
      author={Moran, Shay},
      author={Yehudayoff, Amir},
       title={Sign rank, vc dimension and spectral gaps.},
        date={2014},
   booktitle={Electronic colloquium on computational complexity (eccc)},
      volume={21},
       pages={10},
}

\bib{altan2021estimating}{article}{
      author={Altan, Ege},
      author={Solla, Sara~A},
      author={Miller, Lee~E},
      author={Perreault, Eric~J},
       title={Estimating the dimensionality of the manifold underlying
  multi-electrode neural recordings},
        date={2021},
     journal={PLoS computational biology},
      volume={17},
      number={11},
       pages={e1008591},
}

\bib{bandelt2018coms}{article}{
      author={Bandelt, Hans-J{\"u}rgen},
      author={Chepoi, Victor},
      author={Knauer, Kolja},
       title={Coms: complexes of oriented matroids},
        date={2018},
     journal={Journal of Combinatorial Theory, Series A},
      volume={156},
       pages={195\ndash 237},
}

\bib{basri2009visibility}{inproceedings}{
      author={Basri, Ronen},
      author={Felzenszwalb, Pedro~F},
      author={Girshick, Ross~B},
      author={Jacobs, David~W},
      author={Klivans, Caroline~J},
       title={Visibility constraints on features of 3d objects},
organization={IEEE},
        date={2009},
   booktitle={2009 ieee conference on computer vision and pattern recognition},
       pages={1231\ndash 1238},
}

\bib{bhangale2015complexity}{article}{
      author={Bhangale, Amey},
      author={Kopparty, Swastik},
       title={The complexity of computing the minimum rank of a sign pattern
  matrix},
        date={2015},
     journal={arXiv preprint arXiv:1503.04486},
}

\bib{bjorner1999oriented}{book}{
      author={Bj{\"o}rner, Anders},
      author={Las~Vergnas, Michel},
      author={Sturmfels, Bernd},
      author={White, Neil},
      author={Ziegler, Gunter~M},
       title={Oriented matroids},
   publisher={Cambridge University Press},
        date={1999},
      number={46},
}

\bib{broglia1996lectures}{book}{
      author={Broglia, Fabrizio},
       title={Lectures in real geometry},
   publisher={de Gruyter},
        date={1996},
}

\bib{cunningham2014dimensionality}{article}{
      author={Cunningham, John~P},
      author={Byron, M~Yu},
       title={Dimensionality reduction for large-scale neural recordings},
        date={2014},
     journal={Nature neuroscience},
      volume={17},
      number={11},
       pages={1500\ndash 1509},
}

\bib{curto22novel}{article}{
      author={Curto, Carina},
      author={Lienkaemper, Caitlin},
      author={Alvarez, Juliana~Londono},
      author={Cruz, Hannah Rocio~Santa},
       title={Underlying rank in the presence of monotone nonlinearities},
     journal={In preparation},
}

\bib{curto2021betti}{article}{
      author={Curto, Carina},
      author={Paik, Joshua},
      author={Rivin, Igor},
       title={Betti curves of rank one symmetric matrices},
        date={2021},
     journal={arXiv preprint arXiv:2103.00761},
}

\bib{dunn2018signed}{article}{
      author={Dunn, John~C},
      author={Anderson, Laura},
       title={Signed difference analysis: Testing for structure under
  monotonicity},
        date={2018},
     journal={Journal of Mathematical Psychology},
      volume={85},
       pages={36\ndash 54},
}

\bib{egger20xxtopological}{article}{
      author={Egger, Phillip},
      author={Itskov, Vladimir},
      author={Wu, Min-Chun},
      author={Yarosh, Alex},
       title={Topological detection of monotone rank},
     journal={In preparation},
}

\bib{forster2002linear}{article}{
      author={Forster, J{\"u}rgen},
       title={A linear lower bound on the unbounded error probabilistic
  communication complexity},
        date={2002},
     journal={Journal of Computer and System Sciences},
      volume={65},
      number={4},
       pages={612\ndash 625},
}

\bib{fusi2016neurons}{article}{
      author={Fusi, Stefano},
      author={Miller, Earl~K},
      author={Rigotti, Mattia},
       title={Why neurons mix: high dimensionality for higher cognition},
        date={2016},
     journal={Current opinion in neurobiology},
      volume={37},
       pages={66\ndash 74},
}

\bib{giusti2015clique}{article}{
      author={Giusti, Chad},
      author={Pastalkova, Eva},
      author={Curto, Carina},
      author={Itskov, Vladimir},
       title={Clique topology reveals intrinsic geometric structure in neural
  correlations},
        date={2015},
     journal={Proceedings of the National Academy of Sciences},
      volume={112},
      number={44},
       pages={13455\ndash 13460},
}

\bib{goodman1993allowable}{incollection}{
      author={Goodman, Jacob~E},
      author={Pollack, Richard},
       title={Allowable sequences and order types in discrete and computational
  geometry},
        date={1993},
   booktitle={New trends in discrete and computational geometry},
   publisher={Springer},
       pages={103\ndash 134},
}

\bib{hoffmann2018universality}{article}{
      author={Hoffmann, Udo},
      author={Merckx, Keno},
       title={A universality theorem for allowable sequences with
  applications},
        date={2018},
     journal={arXiv preprint arXiv:1801.05992},
}

\bib{huang2021relationship}{article}{
      author={Huang, Lawrence},
      author={Ledochowitsch, Peter},
      author={Knoblich, Ulf},
      author={Lecoq, Jerome},
      author={Murphy, Gabe~J},
      author={Reid, R~Clay},
      author={de~Vries, Saskia~EJ},
      author={Koch, Christof},
      author={others},
       title={Relationship between simultaneously recorded spiking activity and
  fluorescence signal in gcamp6 transgenic mice},
        date={2021},
     journal={Elife},
      volume={10},
       pages={51675},
}

\bib{husain2020physical}{article}{
      author={Husain, Kabir},
      author={Murugan, Arvind},
       title={{Physical Constraints on Epistasis}},
        date={202005},
        ISSN={0737-4038},
     journal={Molecular Biology and Evolution},
  eprint={https://academic.oup.com/mbe/advance-article-pdf/doi/10.1093/molbev/msaa124/33427979/msaa124.pdf},
         url={https://doi.org/10.1093/molbev/msaa124},
        note={cvaa128},
}

\bib{mnev1988universality}{inproceedings}{
      author={Mn{\"e}v, Nikolai~E},
       title={The universality theorems on the classification problem of
  configuration varieties and convex polytopes varieties},
organization={Springer},
        date={1988},
   booktitle={Topology and geometry—rohlin seminar},
       pages={527\ndash 543},
}

\bib{otwinowski2018biophysical}{article}{
      author={Otwinowski, Jakub},
       title={Biophysical inference of epistasis and the effects of mutations
  on protein stability and function},
        date={2018},
     journal={Molecular biology and evolution},
      volume={35},
      number={10},
       pages={2345\ndash 2354},
}

\bib{padrol2021sweeps}{article}{
      author={Padrol, Arnau},
      author={Philippe, Eva},
       title={Sweeps, polytopes, oriented matroids, and allowable graphs of
  permutations},
        date={2021},
     journal={arXiv preprint arXiv:2102.06134},
}

\bib{paturi1986probabilistic}{article}{
      author={Paturi, Ramamohan},
      author={Simon, Janos},
       title={Probabilistic communication complexity},
        date={1986},
     journal={Journal of Computer and System Sciences},
      volume={33},
      number={1},
       pages={106\ndash 123},
}

\bib{piziak1999full}{article}{
      author={Piziak, Robert},
      author={Odell, PL},
       title={Full rank factorization of matrices},
        date={1999},
     journal={Mathematics magazine},
      volume={72},
      number={3},
       pages={193\ndash 201},
}

\bib{radon1921mengen}{article}{
      author={Radon, Johann},
       title={Mengen konvexer k{\"o}rper, die einen gemeinsamen punkt
  enthalten},
        date={1921},
     journal={Mathematische Annalen},
      volume={83},
      number={1},
       pages={113\ndash 115},
}

\bib{rosen2017convex}{article}{
      author={Rosen, Zvi},
      author={Zhang, Yan~X},
       title={Convex neural codes in dimension 1},
        date={2017},
     journal={arXiv preprint arXiv:1702.06907},
}

\bib{schaefer2009complexity}{inproceedings}{
      author={Schaefer, Marcus},
       title={Complexity of some geometric and topological problems},
organization={Springer},
        date={2009},
   booktitle={International symposium on graph drawing},
       pages={334\ndash 344},
}

\bib{shor1991stretchability}{article}{
      author={Shor, Peter},
       title={Stretchability of pseudolines is np-hard},
        date={1991},
     journal={Applied Geometry and Discrete Mathematics-The Victor Klee
  Festschrift},
}

\bib{siegle2021reconciling}{article}{
      author={Siegle, Joshua~H},
      author={Ledochowitsch, Peter},
      author={Jia, Xiaoxuan},
      author={Millman, Daniel~J},
      author={Ocker, Gabriel~K},
      author={Caldejon, Shiella},
      author={Casal, Linzy},
      author={Cho, Andy},
      author={others},
       title={Reconciling functional differences in populations of neurons
  recorded with two-photon imaging and electrophysiology},
        date={2021},
     journal={Elife},
      volume={10},
       pages={69068},
}

\bib{stringer2019high}{article}{
      author={Stringer, Carsen},
      author={Pachitariu, Marius},
      author={Steinmetz, Nicholas},
      author={Carandini, Matteo},
      author={Harris, Kenneth~D},
       title={High-dimensional geometry of population responses in visual
  cortex},
        date={2019},
     journal={Nature},
      volume={571},
      number={7765},
       pages={361\ndash 365},
}

\bib{sturmfels1987decidability}{article}{
      author={Sturmfels, Bernd},
       title={On the decidability of diophantine problems in combinatorial
  geometry},
        date={1987},
     journal={Bulletin (New series) of the American Mathematical Society},
      volume={17},
      number={1},
       pages={121\ndash 124},
}

\bib{vapnik1971uniform}{article}{
      author={Vapnik, V.~N.},
      author={Chervonenkis, A.~Ya.},
       title={On the uniform convergence of relative frequencies of events to
  their probabilities},
        date={1971},
     journal={Theory of Probability \& Its Applications},
      volume={16},
      number={2},
       pages={264?280},
}

\bib{weinreich2006darwinian}{article}{
      author={Weinreich, Daniel~M},
      author={Delaney, Nigel~F},
      author={DePristo, Mark~A},
      author={Hartl, Daniel~L},
       title={Darwinian evolution can follow only very few mutational paths to
  fitter proteins},
        date={2006},
     journal={science},
      volume={312},
      number={5770},
       pages={111\ndash 114},
}

\end{biblist}
\end{bibdiv}
\bibliographystyle{plain}
\appendix 

\section{Algorithm for recognizing rank two oriented matroids}
\label{sec:rank_2_proof}
To prove Theorem \ref{thm:rank_2_completion}, we show that it is possible to find this order efficiently. We repeat the theorem here: 
\begin{thm*}[Theorem \ref{thm:rank_2_completion}]
Given a set of sign vectors  $\Sigma\subseteq \{+, -\}^n$ of size $m$, there is an $O(mn)$ algorithm to determine whether there is a rank two oriented matroid $\cM$ such that \[\Sigma \subseteq \cT(\cM).\]
\end{thm*} 
We do this with Algorithm \ref{alg:rank_2}. 
	
\begin{algorithm}	
\begin{algorithmic}
\caption{Algorithm for deciding whether a set of topes is contained in a rank two oriented matroid.
\label{alg:rank_2}}
\State \textbf{Input}:  $\Sigma \subseteq \{+, -\}^n$ a set of $m$ sign vectors of length $n$
\State \textbf{Output}: returns true if  $\Sigma$ is contained in the topes of a rank two oriented matroid. 
\Function{IsRank2Topes}{$\Sigma$}
\State Let $\Sigma_{\pm} = \{\pm X \mid X\in \Sigma\}$. 
\State Pick a distinguished element $X^* \in \Sigma_{\pm}$. 
\State Sort $\Sigma_{\pm}$ in increasing order by  $|\sep(X, X^*)|$. Resolve ties arbitrarily. 
\State Let $\Sigma_{final} = [X^*]$. 
\For{ $X\in \Sigma_{\pm}$} 
Let $X_{last}$ be the last element of $\Sigma_{final}$. 
\If{$\sep(X_{last}, X) \cup \sep (X, -X^*) = \sep(X_{last}, -X^*)$} 
\State add $X$ to the end of  $\Sigma_{final}$.  
\EndIf
\EndFor
\If{for each $X\in \Sigma_{\pm}$, either $X$ or $-X$ is contained in $\Sigma_{final}$}
\State\Return true
\Else
\State \Return false 
\EndIf
\EndFunction
\end{algorithmic}
\end{algorithm}

\begin{proof}[Proof of Theorem \ref{thm:rank_2_completion}]
We prove that the function \textsc{IsRank2Topes} from Algorithm \ref{alg:rank_2} returns true if and only if there is a rank two matroid $\cM$ such that 
$X \subseteq \cT(\cM)$. 

First, suppose \textsc{IsRank2Topes}($\Sigma$) returns true. 
Then we can produce a list $\Sigma_{final}$ containing at least one of $X, -X$ in for each $ X\in \Sigma$. 
Further, if we pick $X^*$ as the first element of $\Sigma_{final}$, then for each pair of adjacent elements $X, Y$ in  $\Sigma_{final}$,  $\sep(X,Y) \cup \sep(Y, -X^*) = \sep(X, -X^*)$. 
This implies that for each $i \in [n]$, the sign of the $X_i$ only changes at most once between $X^*$ and $-X^*$. 
We can complete  $\Sigma_{final}$ to a full circular order of all elements of $\Sigma_{\pm}$. 
We do this by appending the list $[-X \mid X\in \Sigma_{final}]$ onto the end of $\Sigma_{final}$, removing the duplicate copies of $\pm X^*$. 
We notice that this gives a new ordering of the entries of $\Sigma_{pm}$ which satisfies the condition given in Lemma  \ref{lem:rank_2_topes} to be the topes of a rank two oriented matroid. 

Next, we prove that if $\Sigma$ is contained within the topes of a rank two oriented matroid, then  \textsc{IsRank2Topes}($\Sigma$) will return true. 
By Lemma  \ref{lem:rank_2_topes}, if $\Sigma$ is contained within the topes of a rank two oriented matroid, then $\Sigma_{\pm}$ can be placed in a circular order such that for any $X$, the sign of $X_i$ changes exactly once between $X$ and $-X$. 
We show that Algorithm \ref{alg:rank_2} will find this order (or its reverse). 
We can pick any distinguished element $X^*$. 
Now, notice that as we go from $X^*$ to $-X^*$ in either direction along this circular order, the quantity $\sep(X, X^*)$ increases. 
Thus, as we try to insert elements into $\Sigma_{final}$, our choice of whichever element to insert first will determine which way we go around the circle, and after this, we will completely fill in one path from $X^*$ to $-X^*$ around the circle. 
Thus, for each $X\in \Sigma_{\pm}$, we will insert either $X$ of $-X$ into $\Sigma_{final}$.

Next, we prove that  Algorithm \ref{alg:rank_2} runs in time $O(mn)$, where $m$ is the number of sign vectors and $n$ is the length of each sign vector. 
For each $X\in \Sigma_{\pm}$, we can calculate $|\sep(X, X^*)|$ in time $O(n)$. We can thus calculate all of these values in $O(mn)$ time. 
Now, since there are only $n$ possibilities for  $|\sep(X, X^*)|$, we can use  bucket-sort to sort $\Sigma_{\pm}$ into increasing order by $|\sep(X, X^*)|$ in $O(m + n)$ time.
We can now loop through this sorted list and check whether  $\sep(X_{last}, X) \cup \sep (X, -X^*) = \sep(X_{last}, -X^*)$ in $O(mn)$ time as well. 
Finally, checking whether we have inserted either $X$ or $-X$ takes $O(m)$ time. 
Thus, overall, the algorithm runs in time $O(mn)$. 
\end{proof}

\end{document}